\documentclass[12pt]{amsart}
\usepackage{amsmath,amsthm,amsfonts,amscd,amssymb,eucal,latexsym,mathrsfs, appendix, yhmath, setspace}
\usepackage[numbers,sort&compress]{natbib}
\usepackage[all,cmtip]{xy}
\usepackage{enumerate}

\newtheorem{theorem}{Theorem}[section]
\newtheorem{corollary}[theorem]{Corollary}
\newtheorem{lemma}[theorem]{Lemma}
\newtheorem{proposition}[theorem]{Proposition}

\theoremstyle{definition}
\newtheorem{definition}[theorem]{Definition}
\newtheorem{remark}[theorem]{Remark}

\newtheorem{example}[theorem]{Example}

\newcommand{\im}{{\rm im}}

\newcommand{\id}{{\rm id}}

\newcommand{\cF}{{\mathcal F}}

\newcommand{\cL}{{\mathcal L}}
\newcommand{\cM}{{\mathcal M}}
\newcommand{\cN}{{\mathcal N}}

\newcommand{\cV}{{\mathcal V}}
\newcommand{\cW}{{\mathcal W}}

\newcommand{\Cb}{{\mathbb C}}

\newcommand{\Fb}{{\mathbb F}}

\newcommand{\Nb}{{\mathbb N}}

\newcommand{\Pb}{{\mathbb P}}

\newcommand{\Rb}{{\mathbb R}}

\newcommand{\Zb}{{\mathbb Z}}
\newcommand{\sA}{{\mathscr A}}
\newcommand{\sB}{{\mathscr B}}

\newcommand{\sF}{{\mathscr F}}

\newcommand{\sM}{{\mathscr M}}

\newcommand{\rL}{{\rm L}}

\newcommand{\End}{{\rm End}}

\newcommand{\Hom}{{\rm Hom}}

\newcommand{\rk}{{\rm rk}}

\newcommand{\coker}{{\rm coker}}
\newcommand{\dom}{{\rm Dom}}
\newcommand{\cod}{{\rm Cod}}

\allowdisplaybreaks

\begin{document}

\title{Bivariant and extended Sylvester rank functions}

\author{Hanfeng Li}

\address{\hskip-\parindent
Center of Mathematics, Chongqing University,
Chongqing 401331, China}
\address{\hskip-\parindent Department of Mathematics, SUNY at Buffalo,
Buffalo, NY 14260-2900, USA}
\email{hfli@math.buffalo.edu}

\date{July 16, 2020}


\begin{abstract}
For a unital ring $R$, a Sylvester rank function is a numerical invariant which can be described in three equivalent ways: on finitely presented left $R$-modules, or on rectangular matrices over $R$, or on maps between finitely generated projective left $R$-modules. We extend each   Sylvester rank function to all pairs of left $R$-modules $\cM_1\subseteq \cM_2$, and to all maps between left $R$-modules satisfying suitable properties including continuity and additivity.

As an application, we show that for any epimorphism $R\rightarrow S$ of unital rings, the pull-back map from the set of Sylvester rank functions of $S$ to that of $R$ is injective. We also give a new proof of Schofield's result describing the image of this map when $S$ is the universal localization of $R$ inverting a set of maps between finitely generated projective left $R$-modules.
\end{abstract}

\maketitle


\section{Introduction} \label{S-introduction}

For a unital ring $R$, a Sylvester rank function for $R$ is a numerical invariant describing the size of finitely presented left $R$-modules. It can be defined in three equivalent ways, all taking values in $\Rb_{\ge 0}$, on either finitely presented left $R$-modules, or rectangular matrices over $R$, or maps between finitely generated projective left $R$-modules (see Section~\ref{S-rank} for the definitions). It was introduced first by Malcolmson in \cite{Malcolmson80} in the first two approaches, and then by Schofield in \cite{Schofield} in the third approach.

Sylvester rank functions arise in many different fields. For a unital $C^*$-algebra $R$, given a tracial state $\tau$ for $R$, one can extend $\tau$ to $M_n(R)$ for all $n\in \Nb$ by setting $\tau(A)=\sum_{j=1}^n\tau(A_{jj})$ for $A\in M_n(R)$, and then define $\rk_\tau(B)=\lim_{k\to \infty}\tau(|B|^{1/k})$ for all $B\in M_{n, m}(R)$. The function $\rk_\tau$ is a Sylvester rank function defined on rectangular matrices over $R$. This rank function is widely studied in Elliott's classification program for simple nuclear $C^*$-algebras, and is fundamental in the definition of strict comparison property and hence in the formulation of the Toms-Winter conjecture \cite{Cuntz, BH, BPT, Winter}.

For a discrete group $\Gamma$, if we take $R$ to be the group von Neumann algebra $\cL\Gamma$, which consists of bounded linear operators on $\ell^2(\Gamma)$ commuting with the right regular representation of $\Gamma$, and take $\tau$ to be the canonical trace given by $\tau(a)=\left< a\delta_{e_\Gamma}, \delta_{e_\Gamma}\right>$, where $\delta_{e_\Gamma}$ is the unit vector in $\ell^2(\Gamma)$ taking value $1$ at the identity element $e_\Gamma$ and $0$ everywhere else, then we obtain the von Neumann rank function on $\cL\Gamma$. This rank function and its restriction on the group algebra $\Cb\Gamma$ play a fundamental role in the definition of $L^2$-Betti numbers \cite{Luck02}.

Systematic study of Sylvester rank functions has also proved useful \cite{AC, Elek16, Elek17, Goodearl}. Such study is vital in recent work of Jaikin-Zapirain on the Atiyah conjecture and the L\"{u}ck approximation conjecture \cite{JZ19}. The classical result of Cohn on epimorphisms of $R$ into division rings \cite{Cohn71} can be stated as that the isomorphism classes of such homomorphisms are in natural $1$-$1$ correspondence with $\Zb_{\ge 0}$-valued Sylvester rank functions on $R$ \cite{Malcolmson80}. This was extended by Schofield to that the equivalence classes of homomorphisms from an algebra $R$ over a field to simple artinian rings, where two such maps are equivalent if the codomains map into a common simple artinian ring $S$ such that the two composition maps from $R$ to $S$ coincide,   are in natural $1$-$1$ correspondence with Sylvester rank functions on $R$ taking value in $\frac{1}{n}\Zb_{\ge 0}$ for some $n\in \Nb$ \cite[Theorem 7.12]{Schofield}.

Sylvester rank functions have been used in the study of direct finiteness. Ara et al. observed in \cite{AOP} that if $R$ has a Sylvester rank function which is faithful in the sense that every nonzero element of $R$ has positive rank, then $R$ is directly finite in the sense that every one-sided invertible element of $R$ is two-sided invertible. They used this observation to show that the group ring $D\Gamma$ is directly finite for any division ring $D$ and any free-by-amenable group $\Gamma$, which is later extended by Elek and Szab\'{o} to all sofic groups \cite{ES04}.

The set of all Sylvester rank functions on $R$ is naturally a compact convex set in a locally convex topological vector space.

Despite the importance of Sylvester rank functions and the nice structure of the set of Sylvester rank functions, in general a Sylvester rank function could have two draw-backs. The first is that it is only defined for finitely presented left $R$-modules or maps between finitely generated projective left $R$-modules. Frequently, we would like it to be defined for all left $R$-modules or maps between all left $R$-modules with suitable properties. The second is that as a measurement of the size of a module, one desirable property for a Sylvester rank function is the additivity, i.e.
for any short exact sequence
$$ 0\rightarrow \cM_1\rightarrow \cM_2\rightarrow \cM_3\rightarrow 0$$
of left $R$-modules, we would like to have that $\dim(\cM_2)=\dim(\cM_1)+\dim(\cM_3)$ if $\dim(\cM_j)$ for $j=1, 2, 3$ are all defined. However, though $\Zb\Gamma$ for every discrete group $\Gamma$ has the restriction of the von Neumann rank, whenever $\Gamma$ is nonamenable, there is a short exact sequence of finitely presented left $\Zb\Gamma$-modules such that the above additivity fails for every Sylvester rank function of $\Zb\Gamma$ (see Example~\ref{E-nonamenable}).

The purpose of this article is to handle these two draw-backs. Given any Sylvester rank function for $R$, we show how to extend it to an invariant for all pairs $(\cM_1, \cM_2)$ of left $R$-modules such that $\cM_1$ is a submodule of $\cM_2$ (Definition~\ref{D-bivariant} and Theorem~\ref{T-extension}). When $\cM_1=\cM_2$ is a finitely presented left $R$-module, we obtain the original Sylvester rank function. This bivariant Sylvester rank function has two desired properties: continuity and additivity (Definition~\ref{D-bivariant} and Theorem~\ref{T-additivity}). Furthermore, the extension is unique.

The bivariant Sylvester rank function can also be described equivalently as an invariant for all maps between left $R$-modules (Definition~\ref{D-map} and Theorem~\ref{T-mod vs map}). The extended Sylvester rank function on all maps also enjoys  continuity and additivity (Definition~\ref{D-map}). However, the full power of additivity is best exhibited at the module level (Theorem~\ref{T-additivity}).

As applications, we apply our construction to study the behaviour of Sylvester rank functions under epimorphisms. Given a unital ring homomorphism $\pi: R\rightarrow S$, one has a natural continuous affine map $\pi^*$ from the space of Sylvester rank functions on $S$ to that on $R$. A natural question is when $\pi^*$ is surjective or injective. We show that if $\pi$ is an epimorphism, then $\pi^*$ is injective (Theorem~\ref{T-epic to injective}). This extends a result of Jaikin-Zapirain in the case $S$ is von Neumann regular. We also describe the image of $\pi^*$ (Theorem~\ref{T-range}). A special case of epimorphism is the map of $R$ to the universal localization ring $R_\Sigma$, where $\Sigma$ is a set of maps between finitely generated projective left $R$-modules. In this case
we give a new proof of the classical result of Schofield describing the image of $\pi^*$ in terms of the rank of elements in $\Sigma$ (Theorem~\ref{T-localization rank}).

This article is organized as follows. In Section~\ref{S-rank} we recall the definitions of Sylvester rank functions. In Section~\ref{S-bivariant} we define the bivariant Sylvester
module rank function, and show that each Sylvester module rank function extends uniquely to a bivariant one. The full additivity of the bivariant Sylvester
module rank function is also established in this section. Section~\ref{S-length} is devoted to discussing when a bivariant Sylvester module rank function is in fact a length function.
The continuity of a bivariant Sylvester module rank function under direct limits is proved in Section~\ref{S-limit}. In Section~\ref{S-map} we define the extended Sylvester map rank function and show that they are in natural $1$-$1$ correspondence with the bivariant Sylvester module rank functions. We also derive various properties of the extended Sylvester map rank functions from those of the bivariant Sylvester module rank functions. Section~\ref{S-induce} is devoted to the study of how an $S$-$R$-bimodule can be used to induce an extended Sylvester map rank function for $R$ from one for $S$. The applications to epimorphisms are given in Section~\ref{S-epic}.

Throughout this article, all modules will be left modules unless specified. All maps between modules will be module homomorphisms. For any module $\cM$, we denote by $\id_\cM$ the identity map of $\cM$. For a map $\alpha:\cM_1\rightarrow \cM_2$ between $R$-modules and an $x\in \cM_1$, we shall write $(x)\alpha$ instead of $\alpha(x)$ for the image of $x$.

\noindent{\it Acknowledgments.}
This work is partially supported by NSF and NSFC grants. It started while I attended the program on $L^2$-invariants at ICMAT in Spring 2018. I am  grateful to Andrei Jaikin-Zapirain for inspiring discussions, especially for suggesting that the bivariant Sylvester module rank function might be used to give a new proof of Schofield's Theorem~\ref{T-localization rank}. I would also like to thank the anonymous referee for very useful comments.

\section{Sylvester Rank Functions} \label{S-rank}

Let $R$ be a unital ring. We recall the definitions and basic facts about Sylvester rank functions for $R$.

\begin{definition} \label{D-module}
A {\it Sylvester module rank function} for $R$ is an $\Rb_{\ge 0}$-valued function $\dim$ on the class of all finitely presented $R$-modules such that
\begin{enumerate}
\item $\dim(0)=0$, $\dim(R)=1$.
\item $\dim(\cM_1\oplus \cM_2)=\dim(\cM_1)+\dim(\cM_2)$.
\item For any exact sequence $\cM_1\rightarrow \cM_2\rightarrow \cM_3\rightarrow 0$, one has
$$\dim(\cM_3)\le \dim(\cM_2)\le \dim(\cM_1)+\dim(\cM_3).$$
\end{enumerate}
\end{definition}
From (3) it is clear that $\dim$ is an isomorphism invariant.

\begin{definition} \label{D-matrix}
A {\it Sylvester matrix rank function} for $R$ is an $\Rb_{\ge 0}$-valued function $\rk$ on the set of all rectangular matrices over $R$ such that
\begin{enumerate}
\item $\rk(0)=0$, $\rk(1)=1$.
\item $\rk(AB)\le \min(\rk(A), \rk(B))$.
\item $\rk(\left[\begin{matrix} A & \\ & B \end{matrix}\right])=\rk(A)+\rk(B)$.
\item $\rk(\left[\begin{matrix} A & C\\ & B \end{matrix}\right])\ge \rk(A)+\rk(B)$.
\end{enumerate}
\end{definition}

The notions of Sylvester module rank functions and Sylvester matrix rank functions were introduced by Malcolmson \cite{Malcolmson80}.

\begin{definition} \label{D-map rank}
A {\it Sylvester map rank function} for $R$ is an $\Rb_{\ge 0}$-valued function $\rk$ on the class of all maps between finitely generated projective $R$-modules
such that
\begin{enumerate}
\item $\rk(0)=0$, $\rk(\id_R)=1$.
\item $\rk(\alpha \beta)\le \min(\rk(\alpha), \rk(\beta))$.
\item $\rk(\left[\begin{matrix} \alpha & \\ & \beta \end{matrix}\right])=\rk(\alpha)+\rk(\beta)$.
\item $\rk(\left[\begin{matrix} \alpha & \gamma\\ & \beta \end{matrix}\right])\ge \rk(\alpha)+\rk(\beta)$.
\end{enumerate}
\end{definition}

The notion of Sylvester map rank functions was introduced by Schofield \cite[page 97]{Schofield}.

\begin{theorem} \label{T-dim vs rank}
There is a natural one-to-one correspondence between Sylvester module rank functions, Sylvester map rank functions, and Sylvester matrix rank functions as follows:
\begin{enumerate}
\item Given a Sylvester module rank function $\dim$, for any map $\alpha:P\rightarrow Q$ between finitely generated projective  $R$-modules $P$ and $Q$,  define $\rk(\alpha)=\dim(Q)-\dim(\coker(\alpha))$. Then $\rk$ is a Sylvester map rank function.

\item Given a Sylvester map rank function $\rk$, for any $A\in M_{n, m}(R)$, consider the map $\alpha_A: R^n\rightarrow R^m$ sending $x$ to $xA$, and define $\rk'(A)=\rk(\alpha_A)$. Then $\rk'$ is a Sylvester matrix rank function.

\item Given a Sylvester matrix rank function $\rk$, for any finitely presented $R$-module $\cM$ take some $A\in M_{n, m}(R)$ such that $\cM\cong R^m/R^nA$, and define $\dim(\cM)=m-\rk(A)$. Then $\dim$ is a Sylvester module rank function.
\end{enumerate}
\end{theorem}

The correspondence between Sylvester module rank functions and Sylvester matrix rank functions in Theorem~\ref{T-dim vs rank} is in \cite[Theorem 4]{Malcolmson80}.
The correspondence between Sylvester module rank functions and Sylvester map rank functions in Theorem~\ref{T-dim vs rank} is in \cite[page 97]{Schofield}.

\begin{example} \label{E-nonamenable}
Let $\Gamma$ be a discrete group. The group ring $R\Gamma$ \cite{Passman} consists of finitely supported functions $f: \Gamma\rightarrow R$ which we shall write as $f=\sum_{s\in \Gamma}f_s s$, where $f_s\in R$ is zero except for finitely many $s\in \Gamma$. The addition and multiplication in $R\Gamma$ are given by
$$\sum_{s\in \Gamma}f_ss+\sum_{s\in \Gamma}g_ss=\sum_{s\in \Gamma}(f_s+g_s)s, \quad \big(\sum_{s\in \Gamma}f_ss\big)\big(\sum_{t\in \Gamma}g_tt\big)=\sum_{s, t\in \Gamma}f_sg_t(st).$$
Now assume that $\Gamma$ is nonamenable, and that $R$ is an integral domain. Denote by $K$ the fractional field of $R$. Then we have the group rings $R\Gamma$ and $K\Gamma$.
Bartholdi showed that for some suitable $n\in \Nb$ there is an injective map $(K\Gamma)^{n+1}\rightarrow (K\Gamma)^n$ of $K\Gamma$-modules \cite{Bartholdi}. Multiplying by a suitable element of $R$, we get an injective map $\alpha: (R\Gamma)^{n+1}\rightarrow (R\Gamma)^n$ of $R\Gamma$-modules, and thus an exact sequence
$$ 0\rightarrow (R\Gamma)^{n+1}\overset{\alpha}{\rightarrow} (R\Gamma)^n\rightarrow \coker(\alpha)\rightarrow 0$$
of finitely presented $R\Gamma$-modules. For any Sylvester module rank function $\dim$ of $R\Gamma$, we have
$$ \dim((R\Gamma)^n)=n<\dim((R\Gamma)^{n+1})+\dim(\coker(\alpha)).$$
\end{example}

\section{Bivariant Sylvester Module Rank Functions} \label{S-bivariant}

Let $R$ be a unital ring.

\begin{definition} \label{D-bivariant}
A {\it bivariant Sylvester module rank function} for $R$ is an $\Rb_{\ge 0}\cup \{+\infty\}$-valued function $(\cM_1, \cM_2)\mapsto \dim(\cM_1|\cM_2)$ on the class of all pairs of $R$-modules $\cM_1\subseteq \cM_2$ satisfying the following conditions:
\begin{enumerate}
\item $\dim(\cM_1|\cM_2)$ is an isomorphism invariant.
\item (Normalization) Setting $\dim(\cM)=\dim(\cM|\cM)$ for all  $R$-modules $\cM$, one has $\dim(0)=0$ and $\dim(R)=1$.
\item (Direct sum) For any $R$-modules $\cM_3\subseteq \cM_4$, one has
$$\dim(\cM_1\oplus \cM_3|\cM_2\oplus \cM_4)=\dim(\cM_1|\cM_2)+\dim(\cM_3|\cM_4).$$
\item (Continuity) $\dim(\cM_1|\cM_2)=\sup_{\cM_1'}\dim(\cM_1'|\cM_2)$ for $\cM_1'$ ranging over all finitely generated $R$-submodules of $\cM_1$.
\item (Continuity) When $\cM_1$ is finitely generated, $\dim(\cM_1|\cM_2)=\inf_{\cM_2'}\dim(\cM_1|\cM_2')$ for $\cM_2'$ ranging over all finitely generated $R$-submodules of $\cM_2$ containing $\cM_1$.
\item (Additivity)
$ \dim(\cM_2)=\dim(\cM_1|\cM_2)+\dim(\cM_2/\cM_1)$.
\end{enumerate}
\end{definition}

\begin{example} \label{E-sofic}
The first bivariant Sylvester module rank function was constructed in \cite{LL19} for the group ring $R\Gamma$ of any sofic group $\Gamma$, when a normalized length function $\rL$ for $R$ (see Definition~\ref{D-normalized length} below) is given. We recall the construction here. The group $\Gamma$ is {\it sofic} \cite{Gromov, Pestov, Weiss} if there is a collection of maps $\Sigma=\{\sigma_j: j\in J\}$ over a directed set $J$ with each $\sigma_j$ being  a map (not necessarily a group homomorphism) from $\Gamma$ to the permutation group of a nonempty finite set $X_j$ such that
\begin{enumerate}
\item for any $s, t\in \Gamma$, one has $\lim_{j\to \infty}\frac{1}{|X_j|}|\{x\in X_j: \sigma_{j, s}\sigma_{j, t}(x)=\sigma_{j, st}(x)\}|=1$,
\item for any distinct $s, t\in \Gamma$, one has $\lim_{j\to \infty}\frac{1}{|X_j|}|\{x\in X_j: \sigma_{j, s}(x)\neq \sigma_{j, t}(x)\}|=1$,
\item $\lim_{j\to \infty}|X_j|=\infty$.
\end{enumerate}
Fix  $\Sigma$ and fix an ultrafilter $\omega$ on
$J$ such that $\omega$ is {\it free} in the sense that for any $j\in J$, the set $\{i \in J : i\ge  j\}$ is in $\omega$.
Let $\cM_1\subseteq \cM_2$ be $R\Gamma$-modules. Denote by $\cF(\Gamma)$ the set of all finite subsets of $\Gamma$, and by $\sF(\cM)$ the set of finitely generated $R$-submodules of any $R\Gamma$-module $\cM$. For any $F\in \cF(\Gamma)$, $\sA\in \sF(\cM_1)$, $\sB\in \sF(\cM_2)$, and $j\in J$, denote by $\sM(\sB, F, \sigma_j)$ the $R$-submodule of $\cM_2^{X_j}$ generated by $\delta_xb-\delta_{\sigma_{j, s}(x)}(sb)$ for all $x\in X_j, s\in F$ and $b\in \sB$, and denote by $\sM(\sA, \sB, F, \sigma_j)$ the image of $\sA^{X_j}$ under the quotient map $\cM_2^{X_j}\rightarrow \cM_2^{X_j}/\sM(\sB, F, \sigma_j)$. Define \cite[Definition 3.1]{LL19}
$$\dim(\cM_1|\cM_2):=\sup_{\sA\in \sF(\cM_1)}\inf_{\sB\in \sF(\cM_2)}\inf_{F\in \cF(\Gamma)}\lim_{j\to \omega}\frac{\rL(\sM(\sA, \sB, F, \sigma_j))}{|X_j|}.$$
Then $\dim(\cdot|\cdot)$ is a bivariant Sylvester module rank function for $R\Gamma$ \cite[Theorem 1.1, Corollary 3.2, Proposition 3.4, Proposition 3.5]{LL19}. For connections of this bivariant Sylvester module rank function to dynamical invariants mean dimension and entropy, see \cite{LL19, Liang}.
If $s\in \Gamma$ has infinite order, then $\dim(R\Gamma (s-1)|R\Gamma)=1$ and $\dim(R\Gamma/R\Gamma(s-1))=0$ \cite[Example 6.3]{LL19}. In particular, when $\Gamma$ is the free group $\Fb_2$ with two generators $s$ and $t$, $R\Fb_2$ has a free $R\Gamma$-submodule with generators $s-1$ and $t-1$ \cite[Corollary 10.3.7.(iv)]{Passman}, and hence $R\Fb_2/R\Fb_2 (s-1)$ contains an $R\Fb_2$-submodule isomorphic to $R\Fb_2$ while $\dim(R\Fb_2/R\Fb_2(s-1))=0$.
\end{example}

For any bivariant Sylvester module rank function $\dim(\cdot|\cdot)$ for $R$, clearly $\cM\mapsto \dim(\cM)$ for finitely presented $R$-modules $\cM$ is a Sylvester module rank function for $R$.

The goal of this section is to prove the following two results.

\begin{theorem} \label{T-extension}
Every Sylvester module rank function for $R$ extends uniquely to a bivariant Sylvester module rank function for $R$.
\end{theorem}

\begin{theorem} \label{T-additivity}
For any bivariant Sylvester module rank function $\dim(\cdot|\cdot)$ for $R$ and any $R$-modules $\cM_1\subseteq \cM_2\subseteq \cM_3$, we have
$$ \dim(\cM_2|\cM_3)=\dim(\cM_1|\cM_3)+\dim(\cM_2/\cM_1|\cM_3/\cM_1).$$
\end{theorem}

From \cite[Lemma 7.7]{LL19} we have the following lemma, which gives the uniqueness part of Theorem~\ref{T-extension}.

\begin{lemma} \label{L-unique}
Let $\dim_1(\cdot|\cdot)$ and $\dim_2(\cdot|\cdot)$ be bivariant Sylvester module rank functions for $R$. If $\dim_1(\cM)=\dim_2(\cM)$ for all finitely presented $R$-modules $\cM$, then $\dim_1=\dim_2$.
\end{lemma}

Let $\dim(\cdot)$ be a Sylvester module rank function for $R$. We shall extend it step by step to a bivariant Sylvester module rank function for $R$.

\begin{lemma} \label{L-quotient fp}
Let $\cM_1$ and $\cM_2$ be finitely presented $R$-modules such that $\cM_1$ is a quotient module of $\cM_2$. Then $\dim(\cM_2)\ge \dim(\cM_1)$.
\end{lemma}
\begin{proof} Let $\alpha: \cM_2\rightarrow \cM_1$ be a surjective map. Then $\ker(\alpha)$ is finitely generated \cite[Proposition 4.26]{Lam}. Thus we have an exact sequence
$$ R^n\rightarrow \cM_2\rightarrow \cM_1\rightarrow 0$$
of finitely presented $R$-modules
for some suitable $n\in \Nb$. Therefore $\dim(\cM_2)\ge \dim(\cM_1)$.
\end{proof}

The following lemma is \cite[Lemma 2]{Malcolmson80}, which is a strengthened version of Schanuel's lemma.

\begin{lemma} \label{L-Schanuel}
Let $\cM_1\subseteq \cM_2$ and $\cM_3\subseteq \cM_4$ be $R$-modules such that $\cM_2, \cM_4$ are projective and $\cM_2/\cM_1\cong \cM_4/\cM_3$. Then there is an isomorphism
$\alpha: \cM_2\oplus \cM_4\rightarrow \cM_2\oplus \cM_4$ such that $(\cM_1\oplus \cM_4)\alpha=\cM_2\oplus \cM_3$.
\end{lemma}

\begin{lemma} \label{L-dim for fg}
Let $\cM$ be a finitely generated $R$-module. Write $\cM$ as $\cM_2/\cM_1$ for some finitely generated projective $R$-module $\cM_2$ and some $R$-submodule $\cM_1$ of $\cM_2$. Then $\inf_{\cM'_1}\dim(\cM_2/\cM_1')$, where $\cM'_1$ runs over finitely generated $R$-submodules of $\cM_1$, does not depend on the choice of the representation of $\cM$ as $\cM_2/\cM_1$. Thus $\dim(\cM):=\inf_{\cM'_1}\dim(\cM_2/\cM_1')=\lim_{\cM_1'\nearrow \cM_1}\dim(\cM_2/\cM_1')$ (where the set of finitely generated $R$-submodules of $\cM_1$ is ordered by inclusion) is well defined, is equal to $\inf_{\cM'} \dim(\cM')$ for $\cM'$ ranging over all finitely presented $R$-modules which admit $\cM$ as a quotient module,  and extends $\dim$ for finitely presented $R$-modules.
\end{lemma}
\begin{proof} Note first that if $\cM_1'\subseteq \cM''_1$ are finitely generated $R$-submodules of $\cM_1$, then $\cM_2/\cM'_1$ and $\cM_2/\cM''_1$ are finitely presented $R$-modules and $\cM_2/\cM''_1$ is a quotient module of $\cM_2/\cM'_1$. Thus by Lemma~\ref{L-quotient fp} we have $\dim(\cM_2/\cM'_1)\ge \dim(\cM_2/\cM''_1)$.

Suppose that we also have $\cM=\cM_4/\cM_3$ for some finitely generated projective $R$-module $\cM_4$. By Lemma~\ref{L-Schanuel} we have an isomorphism $\alpha:\cM_2\oplus \cM_4\rightarrow \cM_2\oplus \cM_4$ such that $(\cM_1\oplus \cM_4)\alpha=\cM_2\oplus \cM_3$.
 Let $\cM'_1$ and $\cM'_3$ be finitely generated $R$-submodules of $\cM_1$ and $\cM_3$ respectively. Then
 $$(\cM_1'\oplus \cM_4)+(\cM_2\oplus \cM'_3)\alpha^{-1}=\cM_1''\oplus \cM_4$$
  for some finitely generated $R$-submodule $\cM_1''$ of $\cM_2$.
Since $\cM_1'\oplus \cM_4\subseteq \cM_1''\oplus \cM_4\subseteq \cM_1\oplus \cM_4$, we have $\cM_1'\subseteq \cM_1''\subseteq \cM_1$. Similarly, we have
$$(\cM_1'\oplus \cM_4)\alpha+(\cM_2\oplus \cM_3')=\cM_2\oplus \cM_3''$$
for some finitely generated $R$-submodule $\cM_3''$ of $\cM_3$ containing $\cM_3'$. Clearly $(\cM_1''\oplus \cM_4)\alpha=\cM_2\oplus \cM_3''$. Therefore $\alpha$ induces an isomorphism $\cM_2/\cM''_1\rightarrow \cM_4/\cM_3''$. From the first paragraph of the proof we then have $\dim(\cM_2/\cM'_1)\ge \dim(\cM_2/\cM''_1)=\dim(\cM_4/\cM_3'')$ and $\dim(\cM_4/\cM'_3)\ge \dim(\cM_4/\cM''_3)=\dim(\cM_2/\cM_1'')$. It follows that $\inf_{\cM'_1}\dim(\cM_2/\cM_1')=\inf_{\cM'_3}\dim(\cM_4/\cM_3')$.

Now it is straightforward to prove the rest of the statements of the lemma.
\end{proof}

The following lemma is obvious.

\begin{lemma} \label{L-quotient fg}
Let $\cM$ be a finitely generated  $R$-module and let $\cM'$ be a quotient module of $\cM$. Then $\dim(\cM)\ge \dim(\cM')$.
\end{lemma}

For any finitely generated  $R$-modules $\cM_1\subseteq \cM_2$, we define $$\dim(\cM_1|\cM_2):=\dim(\cM_2)-\dim(\cM_2/\cM_1)\ge 0.$$

\begin{lemma} \label{L-decreasing}
Let $\cM_1\subseteq \cM_2\subseteq \cM_3$ be finitely generated $R$-modules. Then
$$\dim(\cM_1|\cM_2)\ge \dim(\cM_1|\cM_3).$$
\end{lemma}
\begin{proof} Take finitely generated free $R$-modules $\cN_2\subseteq \cN_3$ and a surjective map $\alpha: \cN_3\rightarrow \cM_3$ with $(\cN_2)\alpha=\cM_2$. Denote by $\cN_1$ the preimage of $\cM_1$ in $\cN_3$. Note that $\cN_1=\ker(\alpha)+(\cN_1\cap \cN_2)$.

Let $\varepsilon>0$. Take finitely generated $R$-submodules $\cN_1'$ and $\cW_2$ of $\cN_1$ and $\ker(\alpha)\cap \cN_2$ respectively such that
$$\dim(\cN_3/\cN_1')-\dim(\cM_3/\cM_1), \dim(\cN_2/\cW_2)-\dim(\cM_2)<\varepsilon.$$
Since $\cN_1=\ker(\alpha)+(\cN_1\cap \cN_2)$, enlarging $\cN_1'$ if necessary, we may assume that $\cN_1'=\cW+\cN'_{12}$ for some finitely generated $R$-submodules $\cW$ and $\cN'_{12}$ of $\ker(\alpha)$ and $\cN_1\cap \cN_2$ respectively such that $\cW_2\subseteq \cW, \cN'_{12}$.

Consider the surjective map $\beta: \cN_3/\cW\oplus \cN_2/\cN'_{12}\rightarrow \cN_3/\cN_1'$ sending
$(x+\cW, y+\cN'_{12})$ to $x-y+\cN'_1$. Also consider the map $\gamma: \cN_2/\cW_2\rightarrow \cN_3/\cW\oplus \cN_2/\cN'_{12}$ sending $z+\cW_2$ to $(z+\cW, z+\cN'_{12})$. Clearly $\gamma\beta=0$. Since $\cW+\cN'_{12}=\cN'_1$, it is easy to see that the sequence
$$ \cN_2/\cW_2\overset{\gamma}{\rightarrow} \cN_3/\cW\oplus \cN_2/\cN'_{12}\overset{\beta}{\rightarrow} \cN_3/\cN_1'\rightarrow 0$$
of finitely presented $R$-modules is exact. Thus
\begin{align*}
\dim(\cM_3)+\dim(\cM_2/\cM_1)&\le \dim(\cN_3/\cW)+\dim(\cN_2/\cN'_{12})\\
&=\dim(\cN_3/\cW\oplus \cN_2/\cN'_{12})\\
&\le \dim(\cN_2/\cW_2)+\dim(\cN_3/\cN_1')\\
&\le \dim(\cM_2)+\dim(\cM_3/\cM_1)+2\varepsilon.
\end{align*}
It follows that
$$\dim(\cM_3)+\dim(\cM_2/\cM_1)\le \dim(\cM_2)+\dim(\cM_3/\cM_1),$$
and hence $\dim(\cM_1|\cM_3)\le \dim(\cM_1|\cM_2)$.
\end{proof}

Let $\cM_2$ be  an $R$-module and $\cM_1$ a finitely generated $R$-submodule of $\cM_2$. We define
$$\dim(\cM_1|\cM_2):=\inf_{\cM_2'}\dim(\cM_1|\cM_2')=\lim_{\cM_2'\nearrow\cM_2}\dim(\cM_1|\cM_2')$$
 for $\cM_2'$ ranging over finitely generated $R$-submodules of $\cM_2$ containing $\cM_1$ ordered by inclusion.
By Lemma~\ref{L-decreasing} it coincides with the earlier definition when both $\cM_1$ and $\cM_2$ are finitely generated.
The following lemma is a direct consequence of Lemma~\ref{L-quotient fg}.

\begin{lemma} \label{L-increasing}
Let $\cM_1\subseteq \cM_2\subseteq \cM_3$ be $R$-modules such that $\cM_1$ and $\cM_2$ are finitely generated. Then $\dim(\cM_1|\cM_3)\le \dim(\cM_2|\cM_3)$.
\end{lemma}

Let $\cM_1\subseteq \cM_2$ be $R$-modules. We define
$$ \dim(\cM_1|\cM_2):=\sup_{\cM_1'}\dim(\cM_1'|\cM_2)=\lim_{\cM_1'\nearrow \cM_1}\dim(\cM_1'|\cM_2)$$
for $\cM_1'$ ranging over finitely generated $R$-submodules of $\cM_1$ ordered by inclusion.
By Lemma~\ref{L-increasing} this extends the earlier definition when $\cM_1$ is finitely generated. We also define
$$\dim(\cM):=\dim(\cM|\cM)$$
for all $R$-modules $\cM$. It coincides with the earlier definition when $\cM$ is finitely generated.

\begin{lemma} \label{L-direct sum}
Let $\cM_1\subseteq \cM_2$ and $\cM_3\subseteq \cM_4$ be $R$-modules. Then
$$\dim(\cM_1\oplus \cM_3|\cM_2\oplus \cM_4)=\dim(\cM_1|\cM_2)+\dim(\cM_3|\cM_4).$$
\end{lemma}
\begin{proof} Let $\cM_1^\sharp$ and $\cM_2^\sharp$ be finitely generated $R$-modules. For $j=1, 2$, write $\cM_j^\sharp$ as $\cM_j^\flat/\cM_j^*$ for some finitely generated free $R$-module $\cM_j^\flat$ and some $R$-submodule $\cM_j^*$ of $\cM_j^\flat$. Then $\cM_1^\sharp\oplus \cM_2^\sharp$ is isomorphic to $(\cM_1^\flat\oplus \cM_2^\flat)/(\cM_1^*\oplus \cM_2^*)$.
We have
\begin{align} \label{E-direct sum1}
\dim(\cM_1^\sharp\oplus \cM_2^\sharp)&=\lim_{\cM'\nearrow (\cM_1^*\oplus \cM_2^*)}\dim\big((\cM_1^\flat\oplus \cM_2^\flat)/\cM'\big)\\
\nonumber &=\lim_{\cM_1'\nearrow \cM_1^*, \cM_2'\nearrow \cM_2^*}\dim\big((\cM_1^\flat\oplus \cM_2^\flat)/(\cM_1'\oplus \cM_2')\big)\\
\nonumber &=\lim_{\cM_1'\nearrow \cM_1^*, \cM_2'\nearrow \cM_2^*}\dim\big((\cM_1^\flat/\cM_1')\oplus (\cM_2^\flat/\cM_2')\big)\\
\nonumber &=\lim_{\cM_1'\nearrow \cM_1^*, \cM_2'\nearrow \cM_2^*}\big(\dim(\cM_1^\flat/\cM_1')+\dim(\cM_2^\flat/\cM_2')\big)\\
\nonumber &=\dim(\cM_1^\sharp)+\dim(\cM_2^\sharp),
\end{align}
where $\cM'$ (resp. $\cM_j'$) ranges over finitely generated $R$-submodules of $\cM_1^*\oplus \cM_2^*$ (resp. $\cM_j^*$)  ordered by inclusion,
and the 4th equality comes from (2) of
Definition~\ref{D-module}.

Now consider the case $\cM_1\subseteq \cM_2$ and $\cM_3\subseteq \cM_4$  are finitely generated $R$-modules. We have
\begin{align*}
\dim(\cM_1\oplus \cM_3|\cM_2\oplus \cM_4)&=\dim(\cM_2\oplus \cM_4)-\dim\big((\cM_2\oplus \cM_4)/(\cM_1\oplus \cM_3)\big)\\
&=\dim(\cM_2\oplus \cM_4)-\dim\big((\cM_2/\cM_1)\oplus (\cM_4/\cM_3)\big)\\
&\overset{\eqref{E-direct sum1}}=\dim(\cM_2)+\dim(\cM_4)-\dim(\cM_2/\cM_1)-\dim(\cM_4/\cM_3)\\
&=\dim(\cM_1|\cM_2)+\dim(\cM_3|\cM_4).
\end{align*}

Next consider the case $\cM_1$ and $\cM_3$ are finitely generated. We have
\begin{align*}
\dim(\cM_1\oplus \cM_3|\cM_2\oplus \cM_4)&= \lim_{\cM'\nearrow (\cM_2\oplus \cM_4)}\dim(\cM_1\oplus \cM_3|\cM')\\
&=\lim_{\cM_2'\nearrow \cM_2, \cM_4'\nearrow \cM_4}\dim(\cM_1\oplus \cM_3|\cM_2'\oplus \cM_4')\\
&=\lim_{\cM_2'\nearrow \cM_2, \cM_4'\nearrow \cM_4}\big(\dim(\cM_1|\cM_2')+\dim(\cM_3|\cM_4')\big)\\
&=\dim(\cM_1|\cM_2)+\dim(\cM_3|\cM_4),
\end{align*}
where $\cM'$ (resp. $\cM_j'$) ranges over finitely generated $R$-submodules of $\cM_2\oplus \cM_4$ (resp. $\cM_j$) containing $\cM_1\oplus \cM_3$ (resp. $\cM_{j-1}$) ordered by inclusion.

Finally consider arbitrary $R$-modules $\cM_1\subseteq \cM_2$ and $\cM_3\subseteq \cM_4$. We have
 \begin{align*}
\dim(\cM_1\oplus \cM_3|\cM_2\oplus \cM_4)&= \lim_{\cM'\nearrow (\cM_1\oplus \cM_3)}\dim(\cM'|\cM_2\oplus \cM_4)\\
&=\lim_{\cM_1'\nearrow \cM_1, \cM_3'\nearrow \cM_3}\dim(\cM_1'\oplus \cM_3'|\cM_2\oplus \cM_4)\\
&=\lim_{\cM_1'\nearrow \cM_1, \cM_3'\nearrow \cM_3}\big(\dim(\cM_1'|\cM_2)+\dim(\cM_3'|\cM_4)\big)\\
&=\dim(\cM_1|\cM_2)+\dim(\cM_3|\cM_4),
\end{align*}
where $\cM'$ (resp. $\cM_j'$) ranges over finitely generated $R$-submodules of $\cM_1\oplus \cM_3$ (resp. $\cM_j$)  ordered by inclusion.
\end{proof}

So far clearly $\dim(\cdot|\cdot)$ satisfies all the conditions in Definition~\ref{D-bivariant} except that the additivity has not been verified yet. In Lemma~\ref{L-additivity} below we shall actually show that $\dim(\cdot|\cdot)$ satisfies the strong additivity in Theorem~\ref{T-additivity}, which then proves both Theorems~\ref{T-extension} and \ref{T-additivity}.

\begin{lemma} \label{L-addition}
For any $R$-modules $\cM_1\subseteq \cM_2\subseteq \cM_3$, if $\cM_1$ is finitely generated, then
\begin{align*}
\dim(\cM_2|\cM_3)=\dim(\cM_1|\cM_3)+\dim(\cM_2/\cM_1|\cM_3/\cM_1).
\end{align*}
\end{lemma}
\begin{proof} If $\cM_1, \cM_2, \cM_3$ are all finitely generated, then
\begin{align*}
\dim(\cM_2|\cM_3)&=\dim(\cM_3)-\dim(\cM_3/\cM_2)\\
&=\dim(\cM_3)-\dim(\cM_3/\cM_1)+\dim(\cM_3/\cM_1)-\dim(\cM_3/\cM_2)\\
&=\dim(\cM_1|\cM_3)+\dim(\cM_2/\cM_1|\cM_3/\cM_1).
\end{align*}

Next when $\cM_1$ and $\cM_2$ are finitely generated, we have
\begin{align*}
\dim(\cM_2|\cM_3)&=\lim_{\cM_3'\nearrow \cM_3}\dim(\cM_2|\cM_3')\\
&=\lim_{\cM_3'\nearrow \cM_3}\dim(\cM_1|\cM_3')+\lim_{\cM_3'\nearrow \cM_3}\dim(\cM_2/\cM_1|\cM_3'/\cM_1)\\
&=\dim(\cM_1|\cM_3)+\dim(\cM_2/\cM_1|\cM_3/\cM_1),
\end{align*}
where $\cM_3'$ ranges over finitely generated $R$-submodules of $\cM_3$ containing $\cM_2$ ordered by inclusion.

Now consider the case $\cM_1$ is finitely generated. We have
\begin{align*}
\dim(\cM_2|\cM_3)&=\sup_{\cM_2'}\dim(\cM_2'|\cM_3)\\
&=\dim(\cM_1|\cM_3)+\sup_{\cM_2'}\dim(\cM_2'/\cM_1|\cM_3/\cM_1)\\
&=\dim(\cM_1|\cM_3)+\dim(\cM_2/\cM_1|\cM_3/\cM_1),
\end{align*}
where $\cM_2'$ ranges over finitely generated $R$-submodules of $\cM_2$ containing $\cM_1$.
\end{proof}

\begin{lemma} \label{L-addition big fg}
For any $R$-modules $\cM_1\subseteq \cM_2$, if $\cM_2$ is finitely generated, then
$$\dim(\cM_2)=\dim(\cM_1|\cM_2)+\dim(\cM_2/\cM_1).$$
\end{lemma}
\begin{proof} Take a surjective map $\alpha: R^m\rightarrow \cM_2$ for some $m\in \Nb$. Denote by $\cM_1^*$ the preimage of $\cM_1$ in $R^m$. Let $\varepsilon>0$.
Take a finitely generated $R$-submodule $\cM_3'$ of $\cM_1^*$ with
\begin{align} \label{E-addition big fg1}
\dim(R^m/\cM_3')\le \dim(\cM_2/\cM_1)+\varepsilon.
\end{align}
Also take a finitely generated $R$-submodule $\cM_1'$ of $\cM_1$ containing $(\cM_3')\alpha$ such that
\begin{align} \label{E-addition big fg2}
 \dim(\cM_1|\cM_2)\le \dim(\cM_1'|\cM_2)+\varepsilon.
 \end{align}
By Lemma~\ref{L-addition} we have
\begin{align} \label{E-addition big fg3}
\dim(\cM_2)=\dim(\cM_1'|\cM_2)+\dim(\cM_2/\cM_1'),
\end{align}
and
\begin{align} \label{E-addition big fg4}
 \dim(\cM_1|\cM_2)=\dim(\cM_1'|\cM_2)+\dim(\cM_1/\cM_1'|\cM_2/\cM_1').
\end{align}
Then
\begin{align*}
\dim(\cM_2)&\overset{\eqref{E-addition big fg3}}{=}\dim(\cM_1'|\cM_2)+\dim(\cM_2/\cM_1')\\
&\ge \dim(\cM_1'|\cM_2)+\dim(\cM_2/\cM_1)\\
&\overset{\eqref{E-addition big fg2}}{\ge} \dim(\cM_1|\cM_2)-\varepsilon+\dim(\cM_2/\cM_1),
\end{align*}
where in the first inequality we apply Lemma~\ref{L-quotient fg}.
We also have
\begin{align*}
\dim(\cM_1|\cM_2)+\dim(\cM_2/\cM_1)&\overset{\eqref{E-addition big fg1}}{\ge} \dim(\cM_1|\cM_2)+\dim(R^m/\cM_3')-\varepsilon \\
&\ge \dim(\cM_1|\cM_2)+\dim(\cM_2/\cM_1')-\varepsilon \\
&\overset{\eqref{E-addition big fg4}}{=} \dim(\cM_1'|\cM_2)+\dim(\cM_1/\cM_1'|\cM_2/\cM_1')\\
& \quad \quad \quad +\dim(\cM_2/\cM_1')-\varepsilon \\
&\overset{\eqref{E-addition big fg3}}{=} \dim(\cM_2)+\dim(\cM_1/\cM_1'|\cM_2/\cM_1')-\varepsilon\\
&\ge \dim(\cM_2)-\varepsilon,
\end{align*}
where in the second inequality we apply Lemma~\ref{L-quotient fg} again.
Letting $\varepsilon\to 0$ we obtain the desired equality.
\end{proof}

\begin{lemma} \label{L-addition big fg2}
For any $R$-modules $\cM_1\subseteq \cM_2\subseteq \cM_3$, if $\cM_3$ is finitely generated, then
$$\dim(\cM_2|\cM_3)=\dim(\cM_1|\cM_3)+\dim(\cM_2/\cM_1|\cM_3/\cM_1).$$
\end{lemma}
\begin{proof} By Lemma~\ref{L-addition big fg} we have
\begin{eqnarray*}
& & \dim(\cM_1|\cM_3)+\dim(\cM_2/\cM_1|\cM_3/\cM_1)\\
&=&\dim(\cM_3)-\dim(\cM_3/\cM_1)+\dim(\cM_3/\cM_1)-\dim(\cM_3/\cM_2)\\
&=&\dim(\cM_3)-\dim(\cM_3/\cM_2)\\
&=&\dim(\cM_2|\cM_3).
\end{eqnarray*}
\end{proof}

\begin{lemma} \label{L-uniform control}
Let $\cM_1\subseteq \cM_2\subseteq \cM_3\subseteq \cM_4$ be $R$-modules such that $\cM_2$ is finitely generated. Then
\begin{align} \label{E-uniform control1}
 \dim(\cM_1|\cM_3)-\dim(\cM_1|\cM_4)\le \dim(\cM_2|\cM_3)-\dim(\cM_2|\cM_4).
 \end{align}
\end{lemma}
\begin{proof}
Since $\cM_2$ is finitely generated, for $j=3, 4$ by Lemma~\ref{L-increasing} we have
$$0\le \dim(\cM_1|\cM_j)\le \dim(\cM_2|\cM_j)\le \dim(\cM_2)<+\infty.$$
Thus all the four dimensions appearing in \eqref{E-uniform control1} are finite.

Consider first the case $\cM_1$ and $\cM_2$  are both finitely generated. By Lemma~\ref{L-addition}  we have
\begin{align*}
\dim(\cM_2|\cM_3)-\dim(\cM_1|\cM_3)&= \dim(\cM_2/\cM_1|\cM_3/\cM_1)\\
&\ge \dim(\cM_2/\cM_1|\cM_4/\cM_1)\\
&=\dim(\cM_2|\cM_4)-\dim(\cM_1|\cM_4),
\end{align*}
and hence
$$ \dim(\cM_1|\cM_3)-\dim(\cM_1|\cM_4)\le \dim(\cM_2|\cM_3)-\dim(\cM_2|\cM_4).$$
Taking limits over finitely generated $R$-submodules of $\cM_1$, we see that the above inequality also holds when $\cM_2$  is finitely generated.
\end{proof}

\begin{proposition} \label{P-fg bound}
For any $R$-modules $\cM_1\subseteq \cM_2\subseteq \cM_3$, if $\cM_2$ is finitely generated, then
$$ \dim(\cM_1|\cM_3)=\inf_{\cM'}\dim(\cM_1|\cM'),$$
where $\cM'$ ranges over finitely generated $R$-submodules of $\cM_3$ containing $\cM_2$.
\end{proposition}
\begin{proof}
Let $\varepsilon>0$.
Take a finitely generated $R$-submodule $\cM'$ of $\cM_3$ containing $\cM_2$ such that
$$\dim(\cM_2|\cM')\le \dim(\cM_2|\cM_3)+\varepsilon.$$
By Lemma~\ref{L-uniform control} we have
$$\dim(\cM_1|\cM')-\dim(\cM_1|\cM_3)\le \dim(\cM_2|\cM')-\dim(\cM_2|\cM_3)\le \varepsilon.$$
\end{proof}

\begin{lemma} \label{L-exact for relative1}
For any exact sequence
$$\cM_1\overset{\alpha}{\rightarrow} \cM_2\rightarrow \cM_3\rightarrow 0$$
of finitely generated $R$-modules, we have $\dim(\cM_2)\le \dim(\cM_1)+\dim(\cM_3)$.
\end{lemma}
\begin{proof}
We have
\begin{align*}
\dim(\cM_2)
&=\dim(\im(\alpha)|\cM_2)+\dim(\cM_3)\\
&\le \dim(\im(\alpha))+\dim(\cM_3)\\
&\le  \dim(\cM_1)+\dim(\cM_3),
\end{align*}
where in the last inequality we apply Lemma~\ref{L-quotient fg}.
\end{proof}

\begin{proposition} \label{P-sum for relative}
For any $R$-modules $\cM_1, \cM_2\subseteq \cM$, we have
\begin{align} \label{E-sum for relative}
\dim(\cM_1+\cM_2|\cM)+\dim(\cM_1\cap \cM_2|\cM)\le \dim(\cM_1|\cM)+\dim(\cM_2|\cM).
\end{align}
\end{proposition}
\begin{proof}
Note that
$$ \dim(\cM_1+\cM_2|\cM)+\dim(\cM_1\cap \cM_2|\cM)=\sup_{\cM_1^\sharp, \cM_2^\sharp}\big(\dim(\cM_1^\sharp+\cM_2^\sharp|\cM)+\dim(\cM_1^\sharp\cap \cM_2^\sharp|\cM)\big)$$
and
$$\dim(\cM_1|\cM)+\dim(\cM_2|\cM)=\sup_{\cM_1^\sharp, \cM_2^\sharp}\big(\dim(\cM_1^\sharp|\cM)+\dim(\cM_2^\sharp|\cM)\big)$$
for $\cM_1^\sharp$ and $\cM_2^\sharp$ ranging over finitely generated $R$-submodules of $\cM_1$ and $\cM_2$ respectively.
Thus it suffices to prove \eqref{E-sum for relative} when $\cM_1$ and $\cM_2$ are finitely generated.
Then by Proposition~\ref{P-fg bound} we may also assume that $\cM$ is finitely generated.

By Lemma~\ref{L-addition big fg} we have
\begin{eqnarray*}
& &\dim(\cM_1+\cM_2|\cM)+\dim(\cM_1\cap \cM_2|\cM)\\
&=&\dim(\cM)-\dim(\cM/(\cM_1+\cM_2))+\dim(\cM)-\dim(\cM/(\cM_1\cap \cM_2)),
\end{eqnarray*}
and
\begin{align*}
\dim(\cM_1|\cM)+\dim(\cM_2|\cM)=\dim(\cM)-\dim(\cM/\cM_1)+\dim(\cM)-\dim(\cM/\cM_2).
\end{align*}
Thus it suffices to show
$$ \dim(\cM/\cM_1)+\dim(\cM/\cM_2)\le \dim(\cM/(\cM_1+\cM_2))+\dim(\cM/(\cM_1\cap \cM_2)).$$
Note that we have the exact sequence
\begin{align*}
0\rightarrow \cM/(\cM_1\cap \cM_2)\overset{\alpha}{\longrightarrow} \cM/\cM_1\oplus \cM/\cM_2
&\overset{\beta}{\longrightarrow} \cM/(\cM_1+\cM_2)\rightarrow 0,
\end{align*}
where $(z+\cM_1\cap \cM_2)\alpha=(z+\cM_1, z+\cM_2)$ and $(x+\cM_1, y+\cM_2)\beta=x-y+\cM_1+\cM_2$.
Then the proposition follows from Lemmas~\ref{L-direct sum} and \ref{L-exact for relative1}.
\end{proof}

\begin{proposition} \label{P-hom for relative}
For any $R$-modules $\cM_1\subseteq \cM_2$ and $\cM$, if $\alpha$ is a  map $\cM_2\rightarrow \cM$, then
$$\dim((\cM_1)\alpha|(\cM_2)\alpha)\le \dim(\cM_1|\cM_2).$$
\end{proposition}
\begin{proof} Clearly it suffices to consider the case $\cM_1$ and $\cM_2$ are both finitely generated. Then $\dim(\ker(\alpha)|\cM_2)\le \dim(\cM_2)<+\infty$.
Applying Proposition~\ref{P-sum for relative} to $\ker(\alpha), \cM_1\subseteq \cM_2$ and using Lemma~\ref{L-addition big fg2}, we have
\begin{align*}
\dim(\cM_1|\cM_2)+\dim(\ker(\alpha)|\cM_2)&\ge \dim(\cM_1+\ker(\alpha)|\cM_2)\\
&=\dim(\ker(\alpha)|\cM_2)+\dim((\cM_1)\alpha|(\cM_2)\alpha),
\end{align*}
and hence $\dim(\cM_1|\cM_2)\ge \dim((\cM_1)\alpha|(\cM_2)\alpha)$.
\end{proof}

\begin{lemma} \label{L-addition10}
For any $R$-modules $\cM_1\subseteq \cM_2\subseteq \cM_3$, we have
\begin{align} \label{E-addition111}
\dim(\cM_2|\cM_3)\le \dim(\cM_1|\cM_3)+\dim(\cM_2/\cM_1|\cM_3/\cM_1).
\end{align}
\end{lemma}
\begin{proof}
Denote by $\alpha$ the quotient map  $\cM_3\rightarrow \cM_3/\cM_1$.

Let $\cM_2'$ be a finitely generated $R$-submodule of $\cM_2$ and let $\cM_3'$ be a finitely generated $R$-submodule of $\cM_3$ containing $\cM_2'$.
Put $\cM^*=\cM_3'\cap \cM_1=\cM_3'\cap \ker(\alpha)$.

Let $\varepsilon>0$.
By Proposition~\ref{P-fg bound} we can find a finitely generated $R$-submodule $\cM_4'$ of $\cM_3$ containing $\cM_3'$ such that
\begin{align} \label{E-addition10}
\dim(\cM^*|\cM_4')\le \dim(\cM^*|\cM_3)+\varepsilon.
\end{align}
We have
\begin{align*}
\dim(\cM_2'|\cM_3)&\le \dim(\cM_2'|\cM_4')\\
&\le \dim(\cM_2'+\cM^*|\cM_4')\\
&=\dim(\cM^*|\cM_4')+\dim((\cM_2'+\cM^*)/\cM^*|\cM_4'/\cM^*)\\
&\le \dim(\cM^*|\cM_4')+\dim((\cM_2'+\cM^*)/\cM^*|\cM_3'/\cM^*)\\
&=\dim(\cM^*|\cM_4')+\dim((\cM_2')\alpha|(\cM_3')\alpha)\\
&\overset{\eqref{E-addition10}}{\le} \dim(\cM^*|\cM_3)+\varepsilon+\dim((\cM_2')\alpha|(\cM_3')\alpha)\\
&\le \dim(\cM_1|\cM_3)+\varepsilon+\dim((\cM_2')\alpha|(\cM_3')\alpha),
\end{align*}
where in the first equality we apply Lemma~\ref{L-addition big fg2}. Letting $\varepsilon\to 0$, we get
$$ \dim(\cM_2'|\cM_3)\le \dim(\cM_1|\cM_3)+\dim((\cM_2')\alpha|(\cM_3')\alpha).$$
Taking infimum over $\cM_3'$,
we obtain
$$\dim(\cM_2'|\cM_3) \le \dim(\cM_1|\cM_3)+\dim((\cM_2')\alpha|\cM_3/\cM_1).$$
Taking supremum over $\cM_2'$, we get \eqref{E-addition111}.
\end{proof}

\begin{lemma} \label{L-additivity}
For any $R$-modules $\cM_1\subseteq \cM_2\subseteq \cM_3$, we have
\begin{align*}
\dim(\cM_2|\cM_3)=\dim(\cM_1|\cM_3)+\dim(\cM_2/\cM_1|\cM_3/\cM_1).
\end{align*}
\end{lemma}
\begin{proof}
By Lemma~\ref{L-addition10} it suffices to show
\begin{align} \label{E-addition1}
\dim(\cM_2|\cM_3)\ge \dim(\cM_1|\cM_3)+\dim(\cM_2/\cM_1|\cM_3/\cM_1).
\end{align}

Let $\cM_1'$ be a finitely generated $R$-submodule of $\cM_1$.
By Proposition~\ref{P-hom for relative} we have
$$ \dim(\cM_2/\cM_1'|\cM_3/\cM_1')\ge \dim(\cM_2/\cM_1|\cM_3/\cM_1).$$
From Lemma~\ref{L-addition} we get
\begin{align*}
\dim(\cM_2|\cM_3)&=\dim(\cM_1'|\cM_3)+\dim(\cM_2/\cM_1'|\cM_3/\cM_1')\\
&\ge \dim(\cM_1'|\cM_3)+\dim(\cM_2/\cM_1|\cM_3/\cM_1).
\end{align*}
Taking supremum over $\cM_1'$, we obtain \eqref{E-addition1}.
\end{proof}

This finishes the proof of Theorems~\ref{T-extension} and \ref{T-additivity}.
In particular, we conclude that Propositions~\ref{P-fg bound}, \ref{P-sum for relative} and \ref{P-hom for relative} hold for any bivariant Sylvester module rank function. In the following proposition we list a few basic properties of bivariant Sylvester module rank functions which are easy consequences of Definition~\ref{D-bivariant} and will be used frequently.

\begin{proposition} \label{P-bivariant basic}
Let $\dim(\cdot|\cdot)$ be a bivariant Sylvester module rank function for $R$. The following hold:
\begin{enumerate}
\item $\dim(\cM_1|\cM_2)$ is increasing in $\cM_1$, i.e. for any $R$-modules $\cM_1\subseteq \cM_1'\subseteq \cM_2$, one has $\dim(\cM_1|\cM_2)\le \dim(\cM_1'|\cM_2)$.
\item     $\dim(\cM_1|\cM_2)$ is decreasing in $\cM_2$, i.e. for any $R$-modules $\cM_1\subseteq \cM_2\subseteq \cM_2'$, one has $\dim(\cM_1|\cM_2)\ge \dim(\cM_1|\cM_2')$.
\item If $\cM_1$ is generated by $n$ elements for some $n\in \Nb$, then $\dim(\cM_1|\cM_2)\le n$.
\end{enumerate}
\end{proposition}

We record the following result which will be used in the proof of Theorem~\ref{T-range}.

\begin{proposition} \label{P-stable}
Let $\dim(\cdot|\cdot)$ be a bivariant Sylvester module rank function for $R$. Let $\cM_1\subseteq \cM_2\subseteq \cM_3\subseteq \cM_4$ be $R$-modules with $\dim(\cM_2|\cM_3)<+\infty$. Then
\begin{align} \label{E-stable}
\dim(\cM_1|\cM_3)-\dim(\cM_1|\cM_4)\le \dim(\cM_2|\cM_3)-\dim(\cM_2|\cM_4).
\end{align}
In particular, if $\dim(\cM_2|\cM_3)=\dim(\cM_2|\cM_4)<+\infty$, then $\dim(\cM_1|\cM_3)=\dim(\cM_1|\cM_4)$.
\end{proposition}
\begin{proof} From Theorem~\ref{T-additivity} we have
\begin{align*}
 \dim(\cM_2|\cM_3)-\dim(\cM_1|\cM_3)&=\dim(\cM_2/\cM_1|\cM_3/\cM_1)\\
 &\ge \dim(\cM_2/\cM_1|\cM_4/\cM_1)\\
 &=\dim(\cM_2|\cM_4)-\dim(\cM_1|\cM_4),
\end{align*}
and hence \eqref{E-stable} holds. If $\dim(\cM_2|\cM_3)=\dim(\cM_2|\cM_4)<+\infty$, then we get $\dim(\cM_1|\cM_3)\le \dim(\cM_1|\cM_4)$, and hence
$\dim(\cM_1|\cM_3)=\dim(\cM_1|\cM_4)$.
\end{proof}

\section{Length Functions} \label{S-length}

Let $R$ be a unital ring. In this section we study the relation between length functions and Sylvester rank functions.

The following is the definition of length function introduced by Northcott and Reufel in \cite{NR}. In fact they consider the general case where $\rL(R)$ could take any value in  $\Rb_{\ge 0}\cup \{+\infty\}$. For relation with the Sylvester rank functions, we require the normalization $\rL(R)=1$ here.

\begin{definition} \label{D-normalized length}
A {\it normalized length function} for $R$ is an $\Rb_{\ge 0}\cup \{+\infty\}$-valued function $\cM\mapsto \rL(\cM)$ on the class of all $R$-modules satisfying the following properties:
\begin{enumerate}
\item (Normalization) $\rL(0)=0$ and $\rL(R)=1$.
\item (Continuity) $\rL(\cM)=\sup_{\cM'}\rL(\cM')$ for $\cM'$ ranging over all finitely generated $R$-submodules of $\cM$.
\item (Additivity) For any short exact sequence $0\rightarrow \cM_1\rightarrow \cM_2\rightarrow \cM_3\rightarrow 0$ of $R$-modules, one has $\rL(\cM_2)=\rL(\cM_1)+\rL(\cM_3)$.
\end{enumerate}
\end{definition}

It follows from the additivity and normalization conditions that each normalized length function is an isomorphism invariant. Clearly the restriction of each normalized length function to the class of finitely presented $R$-modules is a Sylvester module rank function. Furthermore, if $\rL$ is a normalized length function for $R$, then $\dim(\cM_1|\cM_2):=\rL(\cM_1)$ for $R$-modules $\cM_1\subseteq \cM_2$ is a bivariant Sylvester module rank function for $R$.

\begin{proposition} \label{P-length}
Let $\dim(\cdot|\cdot)$ be a bivariant Sylvester module rank function for $R$. The following are equivalent.
\begin{enumerate}
\item $\dim(\cdot)$ is a normalized length function.
\item For any $R$-modules $\cM_1\subseteq \cM_2$ one has $\dim(\cM_1|\cM_2)=\dim(\cM_1)$.
\item For any exact sequence
$$ 0\rightarrow \cM_1\rightarrow \cM_2\rightarrow \cM_3\rightarrow 0$$
of $R$-modules such that $\cM_2$ and $\cM_3$ are finitely presented (then $\cM_1$ must be finitely generated by \cite[Proposition 4.26]{Lam}), one has $\dim(\cM_2)=\dim(\cM_1)+\dim(\cM_3)$.
\end{enumerate}
\end{proposition}
\begin{proof} (2)$\Rightarrow$(1)$\Rightarrow$(3) is trivial.

 (3)$\Rightarrow$(2): Assume that (3) holds. Let $\cM_1\subseteq \cM_2$ be finitely generated $R$-modules. Take a surjective map $\alpha: R^m\rightarrow \cM_2$ for some $m\in \Nb$. Take a finitely generated $R$-submodule $\cM_1^*$ of $R^m$ with $(\cM_1^*)\alpha=\cM_1$. Let $\cM^*$ be a finitely generated $R$-submodule of $\ker(\alpha)$. Then we have the exact sequence
$$ 0\rightarrow (\cM^*+\cM_1^*)/\cM^*\rightarrow R^m/\cM^*\rightarrow R^m/(\cM^*+\cM_1^*)\rightarrow 0,$$
and both $R^m/\cM^*$ and $R^m/(\cM^*+\cM_1^*)$ are finitely presented. Thus
$$ \dim(R^m/\cM^*)=\dim((\cM^*+\cM_1^*)/\cM^*)+\dim(R^m/(\cM^*+\cM_1^*))$$
by (3). Note that $\cM_1$ is a quotient module of $(\cM^*+\cM_1^*)/\cM^*$. Thus $\dim((\cM^*+\cM_1^*)/\cM^*)\ge \dim(\cM_1)$, and hence
$$\dim(R^m/\cM^*)-\dim(R^m/(\cM^*+\cM_1^*))\ge \dim(\cM_1).$$
Then we have
\begin{align*}
\dim(\cM_1)&\ge \dim(\cM_1|\cM_2)\\
&=\dim(\cM_2)-\dim(\cM_2/\cM_1)\\
&=\lim_{\cM^*\nearrow \ker(\alpha)}\dim(R^m/\cM^*)-\lim_{\cM^*\nearrow \ker(\alpha)}\dim(R^m/(\cM^*+\cM_1^*))\\
&\ge \dim(\cM_1),
\end{align*}
where in the third line $\cM^*$ ranges over finitely generated $R$-submodules of $\ker(\alpha)$ ordered by inclusion.
It follows that
$$\dim(\cM_1)=\dim(\cM_1|\cM_2).$$
Taking infimum over finitely generated $R$-submodules of $\cM_2$ containing $\cM_1$, we see that the above equality holds whenever $\cM_1$ is finitely generated.
For any $R$-modules $\cM_1\subseteq \cM_2$, we get
\begin{align*}
\dim(\cM_1)=\dim(\cM_1|\cM_1)=\sup_{\cM_1'}\dim(\cM_1'|\cM_1)=\sup_{\cM_1'}\dim(\cM_1'|\cM_2)=\dim(\cM_1|\cM_2),
\end{align*}
where $\cM_1'$ ranges over finitely generated $R$-submodules of $\cM_1$.
\end{proof}

From Theorem~\ref{T-extension}, Lemma~\ref{L-dim for fg} and Proposition~\ref{P-length} we obtain the following recent result of Virili.

\begin{corollary}[\cite{Virili17}] \label{C-length}
Let $\dim$ be a Sylvester module rank function for $R$. Then $\dim$ extends to a normalized length function for $R$ if and only if for any surjective map $\alpha: \cM_1\rightarrow \cM_2$ of finitely presented $R$-modules one has $\dim(\cM_1)-\dim(\cM_2)=\inf_\cM \dim(\cM)$ for $\cM$ ranging over finitely presented $R$-modules admitting $\ker(\alpha)$ as a quotient module. Furthermore, in such case the extension is unique.
\end{corollary}

If $R$ is von Neumann regular, i.e. for any $x\in R$ there is some $y\in R$ with $xyx=x$, then every finitely presented $R$-module is projective \cite{Goodearl} \cite[Exercise 6.19]{Lam01}, and hence every bivariant Sylvester module rank function for $R$ is a normalized length function.

\section{Direct Limits} \label{S-limit}

In this section we prove Proposition~\ref{P-direct limit}, which gives the continuity of bivariant Sylvester rank functions with respect to direct limits.

Let $R$ be a unital ring and let $\dim(\cdot|\cdot)$ be a bivariant Sylvester module rank function for $R$.

\begin{proposition} \label{P-reduce kernel to fg}
Let $\cM_1\subseteq \cM_2$ be $R$-modules such that $\cM_1$ is finitely generated.
For each $R$-submodule $\cM$ of $\cM_2$ denote by $\gamma_\cM$ the quotient map $\cM_2\rightarrow \cM_2/\cM$. Then for any $R$-submodule $\cM$ of $\cM_2$, we have
$$ \dim((\cM_1)\gamma_{\cM}|\cM_2/\cM)=\inf_{\cM'}\dim((\cM_1)\gamma_{\cM'}|\cM_2/\cM')$$
for $\cM'$ ranging over finitely generated $R$-submodules of $\cM$.
\end{proposition}
\begin{proof} For each $R$-submodule $\cM'$ of $\cM$, since $\gamma_\cM$ factors through $\gamma_{\cM'}$,  by Proposition~\ref{P-hom for relative} we have
$$ \dim((\cM_1)\gamma_\cM |\cM_2/\cM)\le \dim((\cM_1)\gamma_{\cM'}|\cM_2/\cM').$$

Let $\varepsilon>0$.
Take a finitely generated $R$-submodule $\cM^\dag$ of $\cM_2/\cM$ containing $(\cM_1)\gamma_\cM$ such that
$$ \dim((\cM_1)\gamma_\cM |\cM^\dag)\le \dim((\cM_1)\gamma_\cM |\cM_2/\cM)+\varepsilon.$$
Take a finitely generated $R$-submodule $\cM_2^\sharp$ of $\cM_2$ containing $\cM_1$ such that $(\cM_2^\sharp)\gamma_\cM=\cM^\dag$. Since $\dim(\cM_2^\sharp\cap \cM|\cM_2^\sharp)\le \dim(\cM_2^\sharp)<+\infty$,
we can find a finitely generated $R$-submodule $\cM'$ of $\cM_2^\sharp\cap \cM$ such that
$$\dim(\cM_2^\sharp\cap \cM|\cM_2^\sharp)\le \dim(\cM'|\cM_2^\sharp)+\varepsilon.$$
Then by Theorem~\ref{T-additivity} we have
$$ \dim((\cM_2^\sharp\cap \cM)\gamma_{\cM'}|(\cM_2^\sharp)\gamma_{\cM'})=\dim(\cM_2^\sharp\cap \cM|\cM_2^\sharp)-\dim(\cM'|\cM_2^\sharp)\le \varepsilon.$$
Now we have
\begin{eqnarray*}
& &\dim((\cM_1)\gamma_{\cM'}|\cM_2/\cM')\\
&\le& \dim((\cM_1)\gamma_{\cM'}|(\cM_2^\sharp)\gamma_{\cM'})\\
&\le& \dim((\cM_1+(\cM_2^\sharp\cap \cM))\gamma_{\cM'}|(\cM_2^\sharp)\gamma_{\cM'})\\
&=& \dim((\cM_2^\sharp\cap \cM)\gamma_{\cM'}|(\cM_2^\sharp)\gamma_{\cM'})+\dim((\cM_1)\gamma_{\cM_2^\sharp\cap \cM}|(\cM_2^\sharp)\gamma_{\cM_2^\sharp\cap \cM})\\
&=& \dim((\cM_2^\sharp\cap \cM)\gamma_{\cM'}|(\cM_2^\sharp)\gamma_{\cM'})+\dim((\cM_1)\gamma_\cM|(\cM_2^\sharp)\gamma_\cM)\\
&\le& \dim((\cM_1)\gamma_\cM |\cM_2/\cM)+2\varepsilon,
\end{eqnarray*}
where in the first equality we apply Theorem~\ref{T-additivity} again.
\end{proof}

By a {\it direct system of $R$-modules} we mean a family $\{\cM_j\}_{j\in J}$ of $R$-modules indexed by a directed set $J$ and a map $\beta_{jk}: \cM_j\rightarrow \cM_k$ for all $j\le k$ such that $\beta_{jj}=\id_{\cM_j}$ for all $j$ and $\beta_{ij}\beta_{jk}=\beta_{ik}$ for all $i\le j\le k$. For any direct system $(\cM_j, \beta_{jk})$ of $R$-modules over a directed set $J$, one has the direct limit $\underrightarrow{\lim} \cM_k$ \cite[Proposition B-7.7]{Rotman}. For an $R$-module $\cM$, we say that maps $\alpha_j: \cM\rightarrow \cM_j$ for each $j\in J$ and  $\alpha_\infty: \cM\rightarrow \underrightarrow{\lim} \cM_k$ are compatible if $\alpha_j\beta_{jk}=\alpha_k$ for all $j\le k$ and $\alpha_j\beta_{j}=\alpha_\infty$ for all $j\in J$, where $\beta_{j}$ is the canonical map $\cM_j\rightarrow \underrightarrow{\lim} \cM_k$.

\begin{proposition} \label{P-direct limit}
Let $(\cM_j, \beta_{jk})$ be a direct system of $R$-modules over a directed set $J$ with direct limit $\cM_\infty$. Let $\cM$ be an $R$-module with compatible maps $\alpha_j:\cM\rightarrow \cM_j$ and $\alpha_\infty: \cM\rightarrow \cM_\infty$. Suppose that $\dim(\im(\alpha_i)|\cM_i)<+\infty$ for some $i\in J$. Then
$$\dim(\im(\alpha_\infty)|\cM_\infty)=\lim_{j\to \infty}\dim(\im(\alpha_j)|\cM_j)=\inf_{j\in J}\dim(\im(\alpha_j)|\cM_j).$$
\end{proposition}
\begin{proof} From Proposition~\ref{P-hom for relative} we know that $\dim(\im(\alpha_j)|\cM_j)$ decreases. Thus
$$\dim(\im(\alpha_\infty)|\cM_\infty)\le \lim_{j\to \infty}\dim(\im(\alpha_j)|\cM_j)=\inf_{j\in J}\dim(\im(\alpha_j)|\cM_j).$$

Let $\varepsilon>0$. Take a finitely generated $R$-submodule $\cM^\sharp$ of $\cM$ with
\begin{align} \label{E-direct limit1}
\dim(\im(\alpha_i)|\cM_i)\le \dim((\cM^\sharp)\alpha_i|\cM_i)+\varepsilon.
\end{align}
Denote by $\beta_j$ the map $\cM_j\rightarrow \cM_\infty$, and for each submodule $\cM^\dag$ of $\cM_j$ denote by $\gamma_{\cM^\dag}$ the quotient map $\cM_j\rightarrow \cM_j/\cM^\dag$.
Take $j\in J$ with $j\ge i$ such that
\begin{align*}
 \dim((\cM^\sharp)\alpha_\infty|\im(\beta_j))\le \dim((\cM^\sharp)\alpha_\infty|\cM_\infty)+\varepsilon.
\end{align*}
Note that
\begin{align*}
\dim(((\cM^\sharp)\alpha_j)\gamma_{\ker(\beta_j)}|\cM_j/\ker(\beta_j))=\dim((\cM^\sharp)\alpha_\infty|\im(\beta_j))\le \dim((\cM^\sharp)\alpha_\infty|\cM_\infty)+\varepsilon.
\end{align*}
By Proposition~\ref{P-reduce kernel to fg}
we can find a finitely generated $R$-submodule $\cM^\dag$ of $\ker(\beta_j)$ with
\begin{align*}
\dim(((\cM^\sharp)\alpha_j)\gamma_{\cM^\dag}|\cM_j/\cM^\dag)&\le \dim(((\cM^\sharp)\alpha_j)\gamma_{\ker(\beta_j)}|\cM_j/\ker(\beta_j))+\varepsilon\\
&\le \dim((\cM^\sharp)\alpha_\infty|\cM_\infty)+2\varepsilon.
\end{align*}
Take $k\ge j$ such that $(\cM^\dag)\beta_{jk}=0$.
Then $\beta_{jk}$ factors through $\gamma_{\cM^\dag}$. Thus by Proposition~\ref{P-hom for relative} we have
\begin{align*}
\dim(((\cM^\sharp)\alpha_j)\beta_{jk}|\im(\beta_{jk}))\le \dim(((\cM^\sharp)\alpha_j)\gamma_{\cM^\dag}|\cM_j/\cM^\dag)\le \dim((\cM^\sharp)\alpha_\infty|\cM_\infty)+2\varepsilon,
\end{align*}
and hence
\begin{align*}
 \dim((\cM^\sharp)\alpha_k|\cM_k)
&\le \dim((\cM^\sharp)\alpha_k|\im(\beta_{jk}))\\
&= \dim(((\cM^\sharp)\alpha_j)\beta_{jk}|\im(\beta_{jk}))\\
&\le \dim((\cM^\sharp)\alpha_\infty|\cM_\infty)+2\varepsilon.
\end{align*}
Since $\beta_{ik}\gamma_{(\cM^\sharp)\alpha_k}$ factors through $\gamma_{(\cM^\sharp)\alpha_i}$, by Proposition~\ref{P-hom for relative} and Theorem~\ref{T-additivity} we have
\begin{align*}
\dim((\im(\alpha_i))\beta_{ik}\gamma_{(\cM^\sharp)\alpha_k}|(\cM_i)\beta_{ik}\gamma_{(\cM^\sharp)\alpha_k})&\le \dim((\im(\alpha_i))\gamma_{(\cM^\sharp)(\alpha_i)}|(\cM_i)\gamma_{(\cM^\sharp)(\alpha_i)})\\
&=\dim(\im(\alpha_i)|\cM_i)-\dim((\cM^\sharp)\alpha_i|\cM_i)\\
&\overset{\eqref{E-direct limit1}}\le \varepsilon.
\end{align*}
Now by Theorem~\ref{T-additivity} we have
\begin{eqnarray*}
\dim(\im(\alpha_k)|\cM_k)
&=&\dim((\cM^\sharp)\alpha_k|\cM_k)+\dim((\im(\alpha_k))\gamma_{(\cM^\sharp)\alpha_k}|(\cM_k)\gamma_{(\cM^\sharp)\alpha_k})\\
&\le&\dim((\cM^\sharp)\alpha_k|\cM_k)+\dim((\im(\alpha_i))\beta_{ik}\gamma_{(\cM^\sharp)\alpha_k}|(\cM_i)\beta_{ik}\gamma_{(\cM^\sharp)\alpha_k})\\
&\le&\dim((\cM^\sharp)\alpha_\infty|\cM_\infty)+3\varepsilon\\
&\le& \dim(\im(\alpha_\infty)|\cM_\infty)+3\varepsilon.
\end{eqnarray*}
This means
$$\inf_{k\in J}\dim(\im(\alpha_k)|\cM_k)\le \dim(\im(\alpha_\infty)|\cM_\infty)+3\varepsilon.$$
Letting $\varepsilon\to 0$, we finish the proof.
\end{proof}

Note that for any $R$-modules $\cM_1\subseteq \cM_2$, the family $\{\cM_1+\cM^\dag\}$ for $\cM^\dag$ ranging over finitely generated $R$-submodules of $\cM_2$ form a direct system naturally with direct limit $\cM_2$.
The following consequence of Proposition~\ref{P-direct limit} strengthens the continuity condition (5) of Definition~\ref{D-bivariant}.

\begin{corollary} \label{C-direct limit}
Let $\cM_1\subseteq \cM_2$ be $R$-modules. Suppose that $\dim(\cM_1|\cM_1+\cM^\dag)<+\infty$ for some finitely generated $R$-submodule $\cM^\dag$ of $\cM_2$. Then
$$\dim(\cM_1|\cM_2)=\lim_{\cM_2'\nearrow \cM_2}\dim(\cM_1|\cM_1+\cM_2')=\inf_{\cM_2'}\dim(\cM_1|\cM_1+\cM_2')$$
for $\cM_2'$ ranging over finitely generated $R$-submodules of $\cM_2$ ordered by inclusion.
\end{corollary}

\begin{remark}\label{R-direct limit}
The condition $\dim(\im(\alpha_i)|\cM_i)<+\infty$ for some $i\in J$ in Proposition~\ref{P-direct limit} cannot be dropped. For example, take $\cM=\bigoplus_{n\in \Nb}R$, $J=\Nb$, $\cM_j=\bigoplus_{n\ge j}R$ with the maps $\cM\rightarrow \cM_j$ and $\cM_j\rightarrow \cM_k$ for $j\le k$ being natural projections. Then $\cM_\infty=\{0\}$, and hence $\dim(\im(\alpha_\infty)|\cM_\infty)=0$. But $\dim(\im(\alpha_j)|\cM_j)=\dim(\cM_j|\cM_j)=\infty$ for all $j\in \Nb$.

Also, the condition $\dim(\cM_1|\cM_1+\cM^\dag)<+\infty$ for some finitely generated $R$-submodule $\cM^\dag$ of $\cM_2$ in Corollary~\ref{C-direct limit} cannot be dropped.
Suppose that $\cM_1^*\subseteq \cM_2^*$ are finitely generated $R$-modules with $\dim(\cM_1^*)>0$ and $\dim(\cM_1^*|\cM_2^*)=0$ (see Example~\ref{E-sofic} for such an example). Set $\cM_1=\bigoplus_{n\in \Nb}\cM_1^*$ and $\cM_2=\bigoplus_{n\in \Nb}\cM_2^*$.
Then
$$\dim(\bigoplus_{n=1}^m\cM_1^*|\cM_2)\le \dim(\bigoplus_{n=1}^m\cM_1^*|\bigoplus_{n=1}^m\cM_2^*)=\sum_{n=1}^m\dim(\cM_1^*|\cM_2^*)=0$$
for every $m\in \Nb$. Since every finitely generated $R$-submodule of $\cM_1$ is contained in  $\bigoplus_{n=1}^m\cM_1^*$ for some $m\in \Nb$, this shows that $\dim(\cM_1|\cM_2)=0$. For any finitely generated $R$-submodule $\cM_2'$ of $\cM_2$, we have $\cM_2'\subseteq \bigoplus_{n=1}^m\cM_2^*$ for some $m\in \Nb$. Thus
\begin{align*}
 \dim(\cM_1|\cM_1+\cM_2')&\ge \dim(\cM_1|\cM_1+\bigoplus_{n=1}^m\cM_2^*)\\
 &=\dim(\bigoplus_{n>m}\cM_1^*)+\dim(\bigoplus_{n=1}^m\cM_1^*|\bigoplus_{n=1}^m\cM_2^*)\\
 &=+\infty.
 \end{align*}
\end{remark}

\section{Extended Sylvester Map Rank Functions} \label{S-map}

Let $R$ be a unital ring. In this section we introduce extended Sylvester map rank functions and show that they are in natural one-to-one correspondence with bivariant Sylvester module rank functions.

\begin{definition} \label{D-map}
An {\it extended Sylvester map rank function} for $R$ is an $\Rb_{\ge 0}\cup \{+\infty\}$-valued function $\rk$ on the class of all maps between $R$-modules  satisfying the following conditions:
\begin{enumerate}
\item $\rk(0)=0$, $\rk(\id_R)=1$.
\item $\rk(\alpha \beta)\le \min(\rk(\alpha), \rk(\beta))$.
\item $\rk(\left[\begin{matrix} \alpha & \\ & \beta \end{matrix}\right])=\rk(\alpha)+\rk(\beta)$.
\item (Continuity) Let $(\cM_j, \beta_{jk})$ be a direct system of $R$-modules over a directed set $J$ with direct limit $\cM_\infty$. Let $\cM$ be an $R$-module with compatible maps $\alpha_j:\cM_j\rightarrow \cM$ and $\alpha_\infty: \cM_\infty\rightarrow \cM$, i.e. $\beta_{jk}\alpha_k=\alpha_j$ for all $j\le k$ and $\beta_j\alpha_\infty=\alpha_j$ for all $j\in J$, where $\beta_j$ is the canonical map $\cM_j\rightarrow \cM_\infty$. Then
$$\rk(\alpha_\infty)=\lim_{j\to \infty}\rk(\alpha_j).$$
\item (Continuity) Let $(\cM_j, \beta_{jk})$ be a direct system of $R$-modules over a directed set $J$ with direct limit $\cM_\infty$.
Let $\cM$ be a finitely generated $R$-module with compatible maps $\alpha_j:\cM\rightarrow \cM_j$ and $\alpha_\infty: \cM\rightarrow \cM_\infty$. Then
$$\rk(\alpha_\infty)=\lim_{j\to \infty}\rk(\alpha_j).$$
\item (Additivity) For any map $\alpha: \cM_1\rightarrow \cM_2$ between $R$-modules, one has
$$ \rk(\id_{\cM_2})=\rk(\alpha)+\rk(\id_{\coker(\alpha)}).$$
\end{enumerate}
\end{definition}

\begin{theorem} \label{T-mod vs map}
There is a natural $1$-$1$ correspondence between bivariant Sylvester module rank functions for $R$ and extended Sylvester map rank functions for $R$ as follows.
\begin{enumerate}
\item Let $\rk$ be an extended Sylvester map rank function for $R$. For any $R$-modules $\cM_1\subseteq \cM_2$, define $\dim(\cM_1|\cM_2):=\rk(\gamma_{\cM_1\subseteq \cM_2})$, where $\gamma_{\cM_1\subseteq \cM_2}$ denotes the embedding map $\cM_1\hookrightarrow \cM_2$. Then $\dim(\cdot|\cdot)$ is a bivariant Sylvester module rank function for $R$.
\item    Let $\dim(\cdot|\cdot)$ be a bivariant Sylvester module rank function for $R$. For any map $\alpha: \cM_1\rightarrow \cM_2$ of $R$-modules, define $\rk(\alpha):=\dim(\im(\alpha)|\cM_2)$. Then $\rk$ is an extended Sylvester map rank function for $R$.
\end{enumerate}
\end{theorem}
\begin{proof} (1) is trivial. To prove (2), let $\dim(\cdot|\cdot)$ be a bivariant Sylvester module rank function for $R$ and define $\rk$ as in (2).
Conditions (2) and (5) of Definition~\ref{D-map} follow easily from Propositions~\ref{P-hom for relative} and \ref{P-direct limit} respectively.
Thus $\rk$ is an extended Sylvester map rank function for $R$.

If we start with a bivariant Sylvester module rank function $\dim(\cdot|\cdot)$ for $R$, obtain an extended Sylvester map rank function $\rk$ by (2), and then obtain a bivariant Sylvester module rank function $\dim'(\cdot|\cdot)$ by (1) using $\rk$, then clearly $\dim=\dim'$.

Now we start with an extended Sylvester map rank function $\rk$, obtain a bivariant Sylvester module rank function $\dim(\cdot|\cdot)$ by (1), and then obtain an extended Sylvester map rank function $\rk'$ by (2) using $\dim(\cdot|\cdot)$. We need to show that $\rk(\alpha)=\rk'(\alpha)$ for any map $\alpha: \cM_1\rightarrow \cM_2$. Using condition (4) in Definition~\ref{D-map} we may assume that $\cM_1$ is finitely generated. Then using condition (5) in Definition~\ref{D-map} we may assume that $\cM_2$ is also finitely generated.
Note that from conditions (1), (3) and (6) of Definition~\ref{D-map} we know that $\rk(\id_\cM), \rk'(\id_\cM)<+\infty$ for all finitely generated $R$-modules $\cM$.
Then using condition (6) in Definition~\ref{D-map} we may assume that $\cM_1=\cM_2$ is finitely generated and $\alpha=\id_{\cM_2}$. But $\rk(\id_\cM)=\rk'(\id_\cM)$ for any $R$-module $\cM$ follows from the definition of $\rk'$.
\end{proof}

From Theorem~\ref{T-mod vs map} and Lemma~\ref{L-unique} we get immediately

\begin{corollary} \label{C-unique}
Let $\rk_1$ and $\rk_2$ be extended Sylvester map rank functions for $R$. If $\rk_1(\id_{\cM})=\rk_2(\id_{\cM})$ for all finitely presented $R$-modules $\cM$, then $\rk_1=\rk_2$.
\end{corollary}

From Theorems~\ref{T-dim vs rank}, \ref{T-extension} and \ref{T-mod vs map} we may identify Sylvester module rank functions, Sylvester map rank functions, Sylvester matrix rank functions, bivariant Sylvester module rank functions, and extended Sylvester map rank functions. We denote by $\Pb(R)$ the set of all Sylvester rank functions for $R$. Via treating elements of $\Pb(R)$ as Sylvester matrix rank functions equipped with the pointwise convergence topology, $\Pb(R)$ becomes a compact Hausdorff convex subset of a locally convex topological vector space.

For any maps $\alpha: \cM_1\rightarrow \cM_2$ and $\beta: \cM_2\rightarrow \cM_3$ between $R$-modules, we denote the induced map $\coker(\alpha)\rightarrow  \coker(\alpha \beta)$ by $\beta/\alpha$.  From Theorems~\ref{T-additivity} and \ref{T-mod vs map} we obtain

\begin{corollary} \label{C-additivity}
Let $\rk$ be an extended Sylvester map rank function for $R$. For any maps $\alpha: \cM_1\rightarrow \cM_2$ and $\beta: \cM_2\rightarrow \cM_3$ between $R$-modules, we have
$$ \rk(\beta)=\rk(\alpha \beta)+\rk(\beta/\alpha).$$
\end{corollary}

\begin{remark} \label{R-Liang}
In \cite{Liang17} Bingbing Liang  pointed out that the bivariant Sylvester module rank function for $R\Gamma$ in Example~\ref{E-sofic} can be used to define a rank for maps between $R\Gamma$-modules as in Theorem~\ref{T-mod vs map} and that this rank satisfies Corollary~\ref{C-additivity}, though no other properties for this rank were given.
\end{remark}

The following consequence of Proposition~\ref{P-direct limit} strengthens condition (5) in Definition~\ref{D-map}.

\begin{corollary} \label{C-direct limit map}
Let $\rk$ be an extended Sylvester map rank function for $R$.
Let $(\cM_j, \beta_{jk})$ be a direct system of $R$-modules over a directed set $J$ with direct limit $\cM_\infty$.
Let $\cM$ be an $R$-module with compatible maps $\alpha_j:\cM\rightarrow \cM_j$ and $\alpha_\infty: \cM\rightarrow \cM_\infty$. Suppose that $\rk(\alpha_i)<+\infty$ for some $i\in J$. Then
$$\rk(\alpha_\infty)=\lim_{j\to \infty}\rk(\alpha_j)=\inf_{j\in J}\rk(\alpha_j).$$
\end{corollary}

The reader might have noticed that condition (4) in Definition~\ref{D-map rank}  does not appear  in Definition~\ref{D-map}. The next result shows that it is a consequence of the conditions in Definition~\ref{D-map}.

\begin{corollary} \label{C-triangular}
Let $\rk$ be an extended Sylvester map rank function for $R$. For any maps $\alpha: \cM_1\rightarrow \cM_3$, $\beta: \cM_2\rightarrow \cM_4$ and $\gamma: \cM_1\rightarrow \cM_4$ between $R$-modules, we have
$$\rk(\left[\begin{matrix} \alpha & \gamma\\ & \beta \end{matrix}\right])\ge \rk(\alpha)+\rk(\beta).$$
\end{corollary}
\begin{proof} Set $\theta=\left[\begin{matrix} \alpha & \gamma\\ & \beta \end{matrix}\right]: \cM_1\oplus \cM_2\rightarrow \cM_3\oplus \cM_4$. Denote by $\iota$ the embedding $\cM_2\rightarrow \cM_1\oplus \cM_2$.
Note that $\theta/\iota: \cM_1\rightarrow \cM_3\oplus (\cM_4/\im(\beta))$. Denote by
 $p$ the projection $\cM_3\oplus \cM_4\rightarrow \cM_4$, and by $q$ the projection $\cM_3\oplus (\cM_4/\im(\beta))\rightarrow \cM_3$.
 Then $\iota\theta p=\beta$ and $(\theta/\iota)q=\alpha$.
 From the condition (2)  of Definition~\ref{D-map} we have
$\rk(\iota\theta)\ge \rk(\iota\theta p)=\rk(\beta)$ and $\rk(\theta/\iota)\ge \rk((\theta/\iota)q)=\rk(\alpha)$.
Then from Corollary~\ref{C-additivity} we obtain
$$\rk(\theta)=\rk(\iota \theta)+\rk(\theta/\iota)\ge \rk(\beta)+\rk(\alpha).$$
\end{proof}

\section{Induced Rank Functions} \label{S-induce}

In this section we discuss how one extended Sylvester map rank function for one ring induces an extended Sylvester map rank function for another ring via a bimodule.

Let $S$ be a unital ring with an extended Sylvester map rank function $\rk_S$. Let $R$ be a unital ring and let ${}_S\cN_R$ be an $S$-$R$-bimodule with $0<\rk_S(\id_\cN)<+\infty$.
For any map $\alpha: \cM_1\rightarrow \cM_2$ between $R$-modules, we define
\begin{align} \label{E-induced}
f_\cN^*(\rk_S)(\alpha):=\rk_S(\id_\cN\otimes_R\alpha)/\rk_S(\id_\cN).
\end{align}
Since the tensor functor $\cN\otimes_R\cdot$ preserves direct limits \cite[Theorem B-7.15]{Rotman},
using Corollary~\ref{C-direct limit map}
it is easy to conclude that $f_\cN^*(\rk_S)$ is an extended Sylvester map rank function for $R$.

Let $Q$ be a unital ring and ${}_R\cW_Q$  an $R$-$Q$-bimodule. When $0<f_\cN^*(\rk_S)(\id_\cW)=\rk_S(\id_{\cN\otimes_R\cW})/\rk_S(\id_\cN)<+\infty$, we can also define the extended Sylvester map rank functions $f_\cW^*(f_\cN^*(\rk_S))$ and $f_{\cN\otimes_R\cW}^*(\rk_S)$ for $Q$. Clearly we have
$$ f_\cW^*(f_\cN^*(\rk_S))=f_{\cN\otimes_R\cW}^*(\rk_S).$$

Next we consider a few special cases of this construction.

Let $\pi$ be a unital ring homomorphism from $R$ to $S$. Then we may apply the above construction to ${}_SS_R$ and every $\rk_S\in \Pb(S)$. Denoting $f_{{}_SS_R}^*(\rk_S)$ by $\pi^*(\rk_S)$, we obtain  a map $\pi^*: \Pb(S)\rightarrow \Pb(R)$. Explicitly, we have
\begin{align} \label{E-pull back}
\pi^*(\rk_S)(\alpha):=\rk_S(\id_S\otimes_R\alpha)
\end{align}
for any map $\alpha: \cM_1\rightarrow \cM_2$ between $R$-modules.
If we treat $\rk_S$ as a Sylvester matrix rank function for $S$, then clearly
\begin{align} \label{E-pull back rank}
\pi^*(\rk_S)(A)=\rk_S(\pi(A))
\end{align}
for all rectangular matrices $A$ over $R$. Thus $\pi^*$ is continuous and affine.

Conversely, suppose that $\rk_R$ is an extended Sylvester map rank function for $R$, and that $\pi$ is a unital ring homomorphism from $R$ to $S$ such that $0<\rk_R(\id_{S})<+\infty$. (A nontrivial example of this situation is given in Theorem~\ref{T-range} below.)
Then we can apply the above construction to ${}_RS_S$. In this case
\begin{align} \label{E-push forward}
 f_{{}_RS_S}^*(\rk_R)(\beta)=\rk_R(\beta)/\rk_R(\id_{S})
\end{align}
for all maps $\beta: \cN_1\rightarrow \cN_2$ between $S$-modules.

\begin{remark} \label{R-Morita}
When $R$ and $S$ are Morita equivalent unital rings, given Morita equivalence bimodules ${}_R\cW_S$ and ${}_S\cV_R$ \cite[Section 18]{Lam},
there is a natural  homeomorphism between $\Pb(R)$ and $\Pb(S)$ preserving the extremal points as follows.
Note that since ${}_S\cV$ is finitely generated projective and ${}_S\cV^n={}_SS\oplus {}_S\cV'$ for some $n\in \Nb$ and ${}_S\cV'$, we have $0<\rk_S(\id_\cV)<+\infty$ for any extended Sylvester map function $\rk_S$ for $S$. Thus the map $f_{\cV}^*: \Pb(S)\rightarrow \Pb(R)$ is defined and continuous.
Similarly the  map $f_{\cW}^*: \Pb(R)\rightarrow \Pb(S)$ is defined and continuous. Then $f_{\cW}^*f_{\cV}^*=f_{\cV\otimes_R \cW}^*=f_{S}^*$ is the identity map on $\Pb(S)$. Similarly, $f_{\cV}^*f_{\cW}^*$ is the identity map on $\Pb(R)$. Thus $f_{\cV}^*$ and $f_{\cW}^*$ are homeomorphisms and are inverse to each other. Let $\rk
\in \Pb(S)$ be non-extremal. Then $\rk=\lambda_1\rk_1+\lambda_2\rk_2$ for some distinct $\rk_1, \rk_2\in \Pb(S)$ and  $\lambda_1, \lambda_2>0$ with $\lambda_1+\lambda_2=1$. Note that $$f_{\cV}^*(\rk)=\frac{\lambda_1\rk_1(\id_\cV)}{\lambda_1\rk_1(\id_\cV)+\lambda_2\rk_2(\id_\cV)}f_{\cV}^*(\rk_1)+\frac{\lambda_2\rk_2(\id_\cV)}{\lambda_1\rk_1(\id_\cV)+\lambda_2\rk_2(\id_\cV)}f_{\cV}^*(\rk_2).$$
Thus $f_{\cV}^*(\rk)$ is not extremal.
\end{remark}

\section{Epimorphisms} \label{S-epic}

In this section we study the map on Sylvester rank functions induced by epimorphisms.

Let $R$ and $S$ be unital rings. A unital ring homomorphism $\pi: R\rightarrow S$ is called an {\it epimorphism} if for any unital ring $Q$ and any unital ring homomorphisms $\alpha, \beta: S\rightarrow Q$, if $\pi\circ \alpha=\pi\circ \beta$, then $\alpha=\beta$. For example, if $S$ is a division ring and $\im(\pi)$ generates $S$ as a division ring, then $\pi$ is an epimorphism. We refer the reader to \cite[Section XI.1]{Stenstrom} for basic facts about epimorphisms.

\begin{theorem} \label{T-epic to injective}
Let $\pi: R\rightarrow S$ be an epimorphism between unital rings. Let $\rk_S$ be an extended Sylvester map rank function for $S$.
Denote by $\rk_R$
the extended Sylvester map rank function
for $R$ defined via \eqref{E-pull back}. For any map $\alpha: \cN_1\rightarrow \cN_2$ between  $S$-modules, we have
\begin{align} \label{E-injective}
 \rk_S(\alpha)=\rk_R(\alpha).
 \end{align}
In particular, the map $\pi^*: \Pb(S)\rightarrow \Pb(R)$ defined by \eqref{E-pull back} and \eqref{E-pull back rank} is injective.
\end{theorem}
\begin{proof}
Since $\pi$ is an epimorphism,
for any $S$-module $\cN$, the map $S\otimes_R \cN\rightarrow \cN$ sending $a\otimes x$ to $ax$ is an isomorphism of $S$-modules \cite[Proposition XI.1.2]{Stenstrom}.
Thus for any map $\alpha: \cN_1\rightarrow \cN_2$ between $S$-modules, we have
$$\rk_S(\alpha)=\rk_S(\id_S\otimes_R\alpha)=\rk_R(\alpha).$$
\end{proof}

The injectivity part of Theorem~\ref{T-epic to injective} answers a question of Jaikin-Zapirain \cite[Question 5.10]{JZ17} affirmatively and  was proved by him \cite[Proposition 5.11]{JZ19} under the further assumption that $S$ is von Neumann regular, which is vital for his proof of the uniqueness of $*$-regular $R$-algebras associated with a faithful $*$-regular Sylvester matrix rank function for $R$ \cite[Theorem 6.3]{JZ19}.
Note that $S$  may not even be finitely generated as an $R$-module. Thus the formula \eqref{E-injective} does not make sense if we stick to Sylvester map rank functions.

The following result describes the image of $\pi^*$ for epimorphisms $\pi$.

\begin{theorem} \label{T-range}
Let $\pi: R\rightarrow S$ be an epimorphism between unital rings. For any extended Sylvester map rank function  $\rk_R$ for $R$, the following are equivalent:
\begin{enumerate}
\item $\rk_R\in \pi^*(\Pb(S))$.
\item $\rk_R(\id_S\otimes_R \alpha)=\rk_R(\alpha)$ for any map $\alpha: \cM_1\rightarrow \cM_2$ between $R$-modules.
\item $\rk_R(\id_S\otimes_R \id_\cM)=\rk_R(\id_\cM)$ for any finitely presented $R$-module $\cM$.
\item $\rk_R(\pi)=\rk_R(\id_S)=1$.
\end{enumerate}
\end{theorem}
\begin{proof} (1)$\Rightarrow$(2): Assume that $\rk_R=\pi^*(\rk_S)$ for some   extended Sylvester map rank function  $\rk_S$ for $S$. For any  map $\alpha: \cM_1\rightarrow \cM_2$ between $R$-modules, we have
$$ \rk_R(\alpha)\overset{\eqref{E-pull back}}=\rk_S(\id_S\otimes_R\alpha)\overset{\eqref{E-injective}}{=}\rk_R(\id_S\otimes_R\alpha).$$

(2)$\Rightarrow$(3) is trivial.

(3)$\Rightarrow$(1): From (3) we have $\rk_R(\id_S)=\rk_R(\id_S\otimes_R \id_R)=\rk_R(\id_R)=1$. Thus we have the extended Sylvester map rank function $\rk_S:=f_{{}_RS_S}^*(\rk_R)$ for $S$ defined via \eqref{E-push forward}. Then $\rk_S(\alpha)=\rk_R(\alpha)$ for all maps $\alpha$ between $S$-modules.
We are left to show that $\pi^*(\rk_S)=\rk_R$.
Set $\rk_R'=\pi^*(\rk_S)$. Then (3) means $\rk_R'(\id_\cM)=\rk_R(\id_\cM)$ for all finitely presented $R$-modules $\cM$.
From Corollary~\ref{C-unique} we conclude that $\rk_R'=\rk_R$.

(2)$\Rightarrow$(4):
From (2) we have $\rk_R(\id_S)=\rk_R(\id_S\otimes_R\id_R)=\rk_R(\id_R)=1$.
Since $\pi$ is an epimorphism, the natural $S$-bimodule map $S\otimes_RS\rightarrow S$ sending $a\otimes b$ to $ab$ is an isomorphism \cite[Proposition XI.1.2]{Stenstrom}.
Thus by (2) we have $\rk_R(\pi)=\rk_R(\id_S\otimes_R \pi)=\rk_R(\id_S)=1$.

(4)$\Rightarrow$(3): For any $m\in \Nb$ we have $\rk_R(\id_{S^m})=m\rk_R(\id_S)=m$. Let $\cM$ be a finitely presented $R$-module. Write $\cM$ as $\coker(\alpha)$ for some $n,m \in \Nb$ and some map $\alpha: R^n\rightarrow R^m$. Then $S\otimes_R\cM$ is the cokernel of $\id_S\otimes_R\alpha: S^n\cong(S\otimes_RR)^n\rightarrow (S\otimes_RR)^m\cong S^m$. Note that
$$\rk_R(\id_\cM)=\rk_R(\id_{R^m})-\rk_R(\alpha)=m-\rk_R(\alpha),$$
and
$$\rk_R(\id_S\otimes_R\id_\cM)=\rk_R(\id_{S^m})-\rk_R(\id_S\otimes_R \alpha)=m-\rk_R(\id_S\otimes_R \alpha).$$
Thus it suffices to show $\rk_R(\alpha)=\rk_R(\id_S\otimes_R \alpha)$. We have the commutative diagram
\begin{equation*}
\xymatrix
{
& R^n \ar[d]_{\pi^n}  \ar[r]_{\alpha} & R^m  \ar[d]^{\pi^m} \\
& S^n \ar[r]^{\id_S\otimes_R \alpha} & S^m
}
\end{equation*}
Note that $\rk_R(\id_{\coker(\pi)})=\rk_R(\id_S)-\rk_R(\pi)=0$, and hence $\rk_R(\id_{\coker(\pi^n)})=n\rk_R(\id_{\coker(\pi)})=0$. Thus
$\rk_R((\id_S\otimes_S\alpha)/\pi^n)=0$. By Corollary~\ref{C-additivity} we get
$$\rk_R(\id_S\otimes_R \alpha)=\rk_R(\pi^n\circ (\id_S\otimes_R \alpha))+\rk_R((\id_S\otimes_R \alpha)/\pi^n)=\rk_R(\pi^n\circ (\id_S\otimes_R \alpha)).$$
Denote by $\dim_R(\cdot|\cdot)$ the bivariant Sylvester module rank function for $R$ corresponding to $\rk_R$. Then $\dim_R(\im(\pi^m)|S^m)=\rk_R(\pi^m)=m\rk_R(\pi)=m$.
We also have $  \dim_R(\im(\pi^m)|S^m)\le \dim_R(\im(\pi^m))\le \dim_R(R^m)=m$. Thus $\dim_R(\im(\pi^m))=\dim_R(\im(\pi^m)|S^m)=m$.
Then
$$\dim_R(\ker(\pi^m)|R^m)=\dim_R(R^m)-\dim_R(\im(\pi^m))=m-m=0.$$
It follows that
\begin{align*}
\rk_R(\alpha\circ \pi^m)&=\dim_R(\im(\alpha\circ \pi^m)|S^m)\\
&=\dim_R(\im(\alpha\circ \pi^m)|\im(\pi^m))\\
&=\dim_R(\im(\alpha)+\ker(\pi^m)|R^m)-\dim_R(\ker(\pi^m)|R^m)\\
&\ge \dim_R(\im(\alpha)|R^m)=\rk_R(\alpha),
\end{align*}
where the second equality is from Proposition~\ref{P-stable} and the third equality is from Theorem~\ref{T-additivity}. As $\rk_R(\alpha\circ \pi^m)\le \rk_R(\alpha)$, we obtain
$$ \rk_R(\alpha)=\rk_R(\alpha\circ \pi^m)=\rk_R(\pi^n\circ (\id_S\otimes_R \alpha))=\rk_R(\id_S\otimes_R \alpha).$$
\end{proof}

Let $\Sigma$ be a set of maps between finitely generated projective $R$-modules. Denote by $R_\Sigma$ the universal unital ring $S$ with a unital ring homomorphism $\pi: R\rightarrow S$ such that $\id_S\otimes_R\alpha$ is invertible as a map between $S$-modules for every $\alpha\in \Sigma$.
This construction includes the universal localization of $R$ inverting a set of square matrices over $R$ as a special case, but is much more general. For example, given any unital ring homomorphisms $R\rightarrow S$ and $R\rightarrow Q$, if we denote by $S\underset{R}{\cup} Q$ the coproduct (also called the free product) of $S$ and $Q$ amalgamated over $R$,
then $M_2(S\underset{R}{\cup} Q)$ is isomorphic to $(R')_\Sigma$ for $R'=\begin{bmatrix} S &  0 \\  Q\otimes_R S & Q    \end{bmatrix}$ and $\Sigma$ consisting of the map
$\begin{bmatrix} 0 &  0 \\  0 & Q    \end{bmatrix}\rightarrow \begin{bmatrix} S &  0 \\  Q\otimes_R S & 0    \end{bmatrix}$ sending $x$ to $x\begin{bmatrix} 0 &  0 \\  1\otimes 1 & 0    \end{bmatrix}$ for all  $x\in  \begin{bmatrix} 0 &  0 \\  0 & Q    \end{bmatrix}$ \cite[Theorem 4.10]{Schofield}.

The universal localization $R_\Sigma$ was defined via generators and relations in \cite{Bergman74}, from which it is clear that $\pi: R\rightarrow R_\Sigma$ is an epimorphism. Malcolmson gave a  more explicit description of $R_\Sigma$ in the case $\Sigma$ consists of endomorphisms of finitely generated free $R$-modules \cite{Malcolmson82}. In fact his arguments work for general case with minor modification. Denote by $\Sigma'$ the set of maps between finitely generated projective $R$-modules of the form
$$\left[\begin{matrix} \alpha_1 & \cdots &\cdots &\cdots \\ & \alpha_2 & \cdots &\cdots\\  & & \cdots & \cdots \\ & & &\alpha_n   \end{matrix}\right],$$
where each $\alpha_j$ is either in $\Sigma$, or $\id_\cM$ for some finitely generated projective $R$-module $\cM$ appearing as either the domain or codomain of some element in $\Sigma$, or $\id_R$. For any map $\alpha$ between $R$-modules, denote by $\dom(\alpha)$ and $\cod(\alpha)$ the domain and codomain of $\alpha$ respectively. Denote by $\Xi$ the set of all triples
$(f, \alpha, x)$ such that $\alpha\in \Sigma'$, $f$ is a map $R\rightarrow \cod(\alpha)$, and $x$ is a map $\dom(\alpha)\rightarrow R$.  Note that $\id_{R_\Sigma}\otimes_R\alpha$ is invertible as a map between $R_\Sigma$-modules for every $\alpha\in \Sigma'$.

\begin{theorem}[\cite{Malcolmson82}] \label{T-localization}
Every element of $R_\Sigma=\End_{R_\Sigma}(R_\Sigma\otimes_RR)$ is of the form $(\id_{R_\Sigma}\otimes_R f)(\id_{R_\Sigma}\otimes_R\alpha)^{-1}(\id_{R_\Sigma}\otimes_R x)$ for some $(f, \alpha, x)\in \Xi$. Furthermore, for any $(g, \beta, y)\in \Xi$, $(\id_{R_\Sigma}\otimes_R f)(\id_{R_\Sigma}\otimes_R\alpha)^{-1}(\id_{R_\Sigma}\otimes_R x)=(\id_{R_\Sigma}\otimes_R g)(\id_{R_\Sigma}\otimes_R\beta)^{-1}(\id_{R_\Sigma}\otimes_R y)$ if and only if
one has
$$\left[\begin{array}{cccc|c} \alpha & 0 & 0 & 0 & x \\ 0 & \beta & 0 & 0 & -y \\  0 & 0 & \gamma & 0 & 0\\ 0 & 0 & 0 & \theta & w   \\ \hline f & g & h & 0 & 0\end{array}\right]=
\left[\begin{array}{cc} \zeta  \\ \hline u \end{array}\right] \left[\begin{array}{c|c} \eta  & v \end{array}\right] \
$$
for some $\gamma,\theta, \zeta, \eta\in \Sigma'$ and some maps $h: R\rightarrow \cod(\gamma), w: \dom(\theta)\rightarrow R, u:R\rightarrow \cod(\zeta), v: \dom(\eta)\rightarrow R$.
\end{theorem}

The following result is \cite[Theorem 7.4]{Schofield}. Here we use Theorem~\ref{T-range} to give a new proof.

\begin{theorem} \label{T-localization rank}
Let $\rk\in \Pb(R)$ such that $\rk(\alpha)=\rk(\id_{\dom(\alpha)})=\rk(\id_{\cod(\alpha)})$ for every $\alpha\in \Sigma$. Then $R_\Sigma$ is nonzero and $\rk\in \pi^*(\Pb(R_\Sigma))$.
\end{theorem}

Fix $\rk\in \Pb(R)$ such that
\begin{align} \label{E-full}
\rk(\alpha)=\rk(\id_{\dom(\alpha)})=\rk(\id_{\cod(\alpha)})
\end{align}
for all $\alpha\in \Sigma$.
Then clearly \eqref{E-full} holds for all $\alpha\in \Sigma'$.

\begin{lemma} \label{L-localization sum}
For any map $\left[\begin{matrix} \alpha & \gamma \\ 0 & \beta \end{matrix}\right] $ between $R$-modules, if $\alpha\in \Sigma'$ or $\beta\in \Sigma'$, then
$$\rk(\left[\begin{matrix} \alpha & \gamma \\ 0 & \beta \end{matrix}\right])=\rk(\alpha)+\rk(\beta).$$
\end{lemma}
\begin{proof} By Corollary~\ref{C-triangular} we have $\rk(\left[\begin{matrix} \alpha & \gamma \\ 0 & \beta \end{matrix}\right])\ge \rk(\alpha)+\rk(\beta)$. When $\alpha\in \Sigma'$, from
$\left[\begin{matrix} \alpha & \gamma \\ 0 & \beta \end{matrix}\right]=\left[\begin{matrix} \id_{\dom(\alpha)} & 0 \\ 0 & \beta \end{matrix}\right] \left[\begin{matrix} \alpha & \gamma \\ 0 & \id_{\cod(\beta)} \end{matrix}\right]$ we get
$$ \rk(\left[\begin{matrix} \alpha & \gamma \\ 0 & \beta \end{matrix}\right])\le \rk(\left[\begin{matrix} \id_{\dom(\alpha)} & 0 \\ 0 & \beta \end{matrix}\right])=\rk(\id_{\dom(\alpha)})+\rk(\beta)=\rk(\alpha)+\rk(\beta).$$
When $\beta\in \Sigma'$, from
$\left[\begin{matrix} \alpha & \gamma \\ 0 & \beta \end{matrix}\right]=\left[\begin{matrix} \id_{\dom(\alpha)} & \gamma \\ 0 & \beta \end{matrix}\right] \left[\begin{matrix} \alpha & 0 \\ 0 & \id_{\cod(\beta)} \end{matrix}\right]$ we get
$$ \rk(\left[\begin{matrix} \alpha & \gamma \\ 0 & \beta \end{matrix}\right])\le \rk(\left[\begin{matrix} \alpha & 0 \\ 0 & \id_{\cod(\beta)} \end{matrix}\right])=\rk(\alpha)+\rk(\id_{\cod(\beta)})=\rk(\alpha)+\rk(\beta).$$
Therefore $\rk(\left[\begin{matrix} \alpha & \gamma \\ 0 & \beta \end{matrix}\right])=\rk(\alpha)+\rk(\beta)$.
\end{proof}

\begin{lemma} \label{L-kernel}
Let $(f, \alpha, x)\in \Xi$ with $(\id_{R_\Sigma}\otimes_R f)(\id_{R_\Sigma}\otimes_R\alpha)^{-1}(\id_{R_\Sigma}\otimes_R x)=0$. Then
$\rk(\left[\begin{matrix} \alpha & x  \\ f & 0 \end{matrix}\right] )=\rk(\alpha)$.
\end{lemma}
\begin{proof} Since $\alpha$ is a composition of $\left[\begin{matrix} \alpha & x  \\ f & 0 \end{matrix}\right]$ and some other maps, we have $\rk(\left[\begin{matrix} \alpha & x  \\ f & 0 \end{matrix}\right] )\ge \rk(\alpha)$.
By Theorem~\ref{T-localization} we have
$$\left[\begin{array}{cccc|c} \alpha & 0 & 0 & 0 & x \\ 0 & \id_R & 0 & 0 & 0 \\  0 & 0 & \gamma & 0 & 0\\ 0 & 0 & 0 & \theta & w   \\ \hline f & 0 & h & 0 & 0\end{array}\right]=\left[\begin{array}{cc} \zeta  \\ \hline u \end{array}\right] \left[\begin{array}{c|c} \eta  & v \end{array}\right] $$
for some $\gamma,\theta, \zeta, \eta\in \Sigma'$ and some maps $h: R\rightarrow \cod(\gamma), w: \dom(\theta)\rightarrow R, u:R\rightarrow \cod(\zeta), v: \dom(\eta)\rightarrow R$.
From Lemma~\ref{L-localization sum} we have
$$ \rk({\rm LHS})=\rk(\left[\begin{matrix} \alpha & x  \\ f & 0 \end{matrix}\right])+\rk(\id_R)+\rk(\gamma)+\rk(\theta).$$
Note that
\begin{align*}
\rk({\rm RHS})&\le \rk(\left[\begin{matrix} \zeta  \\ u \end{matrix}\right])\le \rk(\id_{\cod(\zeta)})=\rk(\id_{\dom(\zeta)})\\
&=\rk(\id_{\dom(\alpha)})+\rk(\id_R)+\rk(\id_{\dom(\gamma)})+\rk(\id_{\dom(\theta)})\\
&=\rk(\alpha)+\rk(\id_R)+\rk(\gamma)+\rk(\theta).
\end{align*}
Since $\rk({\rm LHS})=\rk({\rm RHS})$, we get
 $\rk(\left[\begin{matrix} \alpha & x  \\ f & 0 \end{matrix}\right])\le \rk(\alpha)$. This finishes the proof.
\end{proof}

We are ready to prove Theorem~\ref{T-localization rank}.

\begin{proof}[Proof of Theorem~\ref{T-localization rank}]
Note that $(\id_R, \id_R, \id_R)\in \Xi$. Since $\rk(\left[\begin{matrix} \id_R & \id_R  \\ \id_R & 0 \end{matrix}\right])=2>\rk(\id_R)$, from Lemma~\ref{L-kernel} we have
$1_{R_\Sigma}\neq 0$. Thus $R_\Sigma$ is nonzero.

By Theorem~\ref{T-range} we just need to show $\rk(\pi)=\rk(\id_{R_\Sigma})=1$. Since $\rk(\pi)\le \rk(\id_{R_\Sigma})$, it suffices to show $\rk(\pi)\ge 1$ and $\rk(\id_{R_\Sigma})\le 1$. Denote by $\dim(\cdot|\cdot)$ the bivariant Sylvester module rank function for $R$ corresponding to $\rk$.

Let $\cM$ be a finitely generated $R$-submodule of $R_\Sigma$. Say, $\cM$ is generated by $(\id_{R_\Sigma}\otimes_R f_j)(\id_{R_\Sigma}\otimes_R\alpha_j)^{-1}(\id_{R_\Sigma}\otimes_R x_j)$ with $(f_j, \alpha_j, x_j)\in \Xi$ for $j=1, \dots, n$. Set
$$ \theta=\left[\begin{matrix} \alpha_1 &  & & \\ & \alpha_2 &  & \\  & & \ddots &  \\ & & &\alpha_n   \end{matrix}\right]\in \Sigma', \quad \beta=\left[\begin{matrix} \theta  &  \\ & \id_R \end{matrix}\right]\in \Sigma', $$
and
$$z=\left[\begin{matrix} x_1 \\ x_2 \\  \vdots \\ x_n   \end{matrix}\right]: \dom(\theta)\rightarrow R,\quad  y=\left[\begin{matrix} z\\ \id_R \end{matrix}\right]: \dom(\beta)\rightarrow R.$$
Note that we may identity $\cM'$ with $\Hom_R(R, \cM')$ for any $R$-module $\cM'$.  Consider the $R$-module  map
$\gamma: \cod(\beta)\rightarrow R_\Sigma$ sending $g$ to $(\id_{R_\Sigma}\otimes_R g)(\id_{R_\Sigma}\otimes_R\beta)^{-1}(\id_{R_\Sigma}\otimes_R y)$.
Then $\im(\gamma)$ is a finitely generated $R$-submodule of $R_\Sigma$ containing $\cM+\im(\pi)$.

We claim that $\dim(\ker(\gamma)|\cod(\beta))=\rk(\theta)$.
Let $\cM^\sharp$ be a finitely generated $R$-submodule of $\ker(\gamma)$.
Say, $\cM^\sharp$ is generated by
$g_j\in \ker(\gamma)$ for $j=1, \dots, m$.  For each $1\le j\le m$, by Lemma~\ref{L-kernel} we have
$$\dim(\im(\left[\begin{matrix} \beta & y  \\ g_j & 0 \end{matrix}\right])|\cod(\beta)\oplus R)=\rk(\left[\begin{matrix} \beta & y  \\ g_j & 0 \end{matrix}\right])=\rk(\beta).$$
For any $1\le k<m$, note that
\begin{align*}
\dim((\sum_{j=1}^k\im(\left[\begin{matrix} \beta & y  \\ g_j & 0 \end{matrix}\right]))\cap \im(\left[\begin{matrix} \beta & y  \\ g_{k+1} & 0 \end{matrix}\right])|\cod(\beta)\oplus R)&\ge \dim(\im(\left[\begin{matrix} \beta & y  \end{matrix}\right])|\cod(\beta)\oplus R)\\
&=\rk(\left[\begin{matrix} \beta & y  \end{matrix}\right])\ge \rk(\beta).
\end{align*}
Then from Proposition~\ref{P-sum for relative} we get
$$ \dim(\sum_{j=1}^{k+1}\im(\left[\begin{matrix} \beta & y  \\ g_j & 0 \end{matrix}\right])|\cod(\beta)\oplus R)\le \dim(\sum_{j=1}^k\im(\left[\begin{matrix} \beta & y  \\ g_j & 0 \end{matrix}\right])|\cod(\beta)\oplus R).$$
It follows that
$$ \dim(\sum_{j=1}^m\im(\left[\begin{matrix} \beta & y  \\ g_j & 0 \end{matrix}\right])|\cod(\beta)\oplus R)\le \rk(\beta).$$
Denote by $\zeta$ the quotient map $\cod(\beta)\oplus R\rightarrow (\cod(\beta)\oplus R)/\cM^\sharp$, and by $\eta$ the projection $\cod(\beta)\oplus R\rightarrow R$. Then $\eta$ factors through $\zeta$. Consider $(1, 1)\in R\oplus R\subseteq \cod(\theta)\oplus R\oplus R=\cod(\beta)\oplus R$. Note that
$\sum_{j=1}^m\im(\left[\begin{matrix} \beta & y  \\ g_j & 0 \end{matrix}\right])\supseteq \cM^\sharp+R(1, 1)$.
Now we get
\begin{align*}
\rk(\theta)+1 &=\rk(\beta)\\
&\ge \dim(\sum_{j=1}^m\im(\left[\begin{matrix} \beta & y  \\ g_j & 0 \end{matrix}\right])|\cod(\beta)\oplus R)\\
&\ge \dim(\cM^\sharp+R(1, 1)|\cod(\beta)\oplus R)\\
&=\dim(\cM^\sharp|\cod(\beta)\oplus R)+\dim((R(1, 1))\zeta|\im(\zeta))\\
&\ge \dim(\cM^\sharp|\cod(\beta)\oplus R)+\dim((R(1, 1))\eta|\im(\eta))\\
&=\dim(\cM^\sharp|\cod(\beta))+\dim(R|R)\\
&=\dim(\cM^\sharp|\cod(\beta))+1,
\end{align*}
where in the second equality we apply Theorem~\ref{T-additivity} and in the last inequality we apply Proposition~\ref{P-hom for relative}. Therefore $\dim(\cM^\sharp|\cod(\beta))\le \rk(\theta)$.
Taking supremum over $\cM^\sharp$ we get $\dim(\ker(\gamma)|\cod(\beta))\le \rk(\theta)$.
Consider the map $\left[\begin{matrix} \theta & -z  \end{matrix}\right]: \dom(\theta)\rightarrow \cod(\theta)\oplus R=\cod(\beta)$.
Clearly $\left[\begin{matrix} \theta & -z  \end{matrix}\right]\gamma=0$, and hence $\im(\left[\begin{matrix} \theta & -z  \end{matrix}\right])\subseteq \ker(\gamma)$.
Thus
$$ \dim(\ker(\gamma)|\cod(\beta))\ge \dim(\im(\left[\begin{matrix} \theta & -z  \end{matrix}\right])|\cod(\beta))=\rk(\left[\begin{matrix} \theta & -z  \end{matrix}\right])\ge \rk(\theta).$$
This proves our claim.

Now we have
\begin{align*}
\dim(\ker(\gamma)+R|\cod(\beta))&\ge \dim(\im(\left[\begin{matrix} \theta & -z  \end{matrix}\right])+R|\cod(\beta))\\
&=\dim(\im(\left[\begin{matrix} \theta & -z  \\ & \id_R\end{matrix}\right])|\cod(\beta))\\
&=\rk(\left[\begin{matrix} \theta & -z  \\ & \id_R\end{matrix}\right])=\rk(\theta)+1.
\end{align*}
By Theorem~\ref{T-additivity} we have
\begin{align*}
\dim(\im(\pi)|\cM+\im(\pi))&\ge \dim(\im(\pi)|\im(\gamma))\\
&=\dim(\ker(\gamma)+R|\cod(\beta))-\dim(\ker(\gamma)|\cod(\beta))\\
&\ge (\rk(\theta)+1)-\rk(\theta)=1.
\end{align*}
Taking infimum over $\cM$, we obtain $\rk(\pi)=\dim(\im(\pi)|R_\Sigma)\ge 1$.

We also have
\begin{align*}
\dim(\cM|R_\Sigma)&\le \dim(\im(\gamma)|R_\Sigma)\le \dim(\im(\gamma))\\
&=\dim(\cod(\beta))-\dim(\ker(\gamma)|\cod(\beta))\\
&=\rk(\beta)-\rk(\theta)=1.
\end{align*}
Taking supremum over $\cM$, we obtain $\rk(\id_{R_\Sigma})=\dim(R_\Sigma)\le 1$ as desired.
\end{proof}


\end{document}